\documentclass[a4paper,11pt]{amsart}
\usepackage{palatino, verbatim}
\usepackage[latin1]{inputenc}
\usepackage[T1]{fontenc}
\usepackage{amsthm, amsfonts}
\usepackage{amsfonts}
\usepackage{graphicx}
\usepackage{amssymb}
\usepackage{amsmath}
\usepackage{latexsym}
\usepackage[all]{xy}
\usepackage{mathrsfs}

\newtheorem{thm}{Theorem}[section]
\newtheorem{cor}[thm]{Corollary}
\newtheorem{lem}[thm]{Lemma}
\newtheorem{prop}[thm]{Proposition}
\newtheorem{conj}[thm]{Conjecture}

\theoremstyle{definition}
\newtheorem{df}{Definition}
\newtheorem{rem}{Remark}

\newtheorem{ex}{Example}
\newtheorem{cex}{Counterexample}

          \newcommand\im{\mathrm{Im}}
          \newcommand\Pic{\mathrm{Pic}}
          \newcommand\W{\mathcal W}
          \newcommand\oo{\mathcal O}
          
          \newcommand\Ext{\mathrm{Ext}}
          
          \newcommand\Hom{\mathrm{Hom}}
          \newcommand\Cliff{\mathrm{Cliff}}
          \newcommand\rk{\mathrm{rk}}

\newcounter{appendice}
\def\theappendice{{\normalsize\Alph{appendice}}}
\newcommand{\appendice}[1]
{
\refstepcounter{appendice}
\def\theequation{\theappendice.\arabic{equation}}
\def\thesection{{\normalsize\Alph{appendice}}}
\bigskip\bigskip\noindent
\begin{center}
{\Large\textbf{Appendix.\ #1}}
\end{center}
\par\bigskip\nopagebreak
}

\begin{document}

\title[Generalized Lazarsfeld-Mukai bundles and a conjecture of Donagi and Morrison]{Generalized Lazarsfeld-Mukai bundles and a conjecture of Donagi and Morrison}
\author{Margherita Lelli--Chiesa}
\dedicatory{ (with an Appendix joint with \textsc{Andreas Leopold Knutsen})}
\address{Max Planck Institut for Mathematics, 53111 Bonn}
\email{lelli@mpim-bonn.mpg.de}
%\author{a joint appendix with Andreas Knutsen}
%\contrib[and a joint appendix with]{Andreas Knutsen}
\begin{abstract}
Let $S$ be a $K3$ surface and assume for simplicity that it does not contain any ($-2$)-curve. Using coherent systems, we express every non-simple Lazarsfeld-Mukai bundle on $S$ as an extension of two sheaves of some special type, that we refer to as {\em generalized Lazarsfeld-Mukai bundles}. This has interesting consequences concerning the Brill-Noether theory of curves $C$ lying on $S$. From now on, let $g$ denote the genus of $C$ and $A$ be a complete linear series of type $g^r_d$ on $C$ such that $d\leq g-1$ and the corresponding  Brill-Noether number is negative. First, we focus on the cases where $A$ computes the Clifford index; if $r>1$ and with only some completely classified exceptions, we show that $A$ coincides with the restriction to $C$ of a line bundle on $S$. This is a refinement of Green and Lazarsfeld's result on the constancy of the Clifford index of curves moving in the same linear system. Then, we study a conjecture of Donagi and Morrison predicting that, under no hypothesis on its Clifford index, $A$ is contained in a $g^s_e$ which is cut out from a line bundle on $S$ and satisfies $e\leq g-1$. We provide counterexamples to the last inequality already for $r=2$. A slight modification of the conjecture, which holds for $r=1,2$, is proved under some hypotheses on the pair $(C,A)$ and its deformations. We show that the result is optimal (in the sense that our hypotheses cannot be avoided) by exhibiting, in the Appendix, some counterexamples obtained jointly with Andreas Leopold Knutsen.
\end{abstract}

\maketitle
\section{Introduction}
The use of Lazarsfeld-Mukai bundles (LM bundles for short) in the study of the Brill-Noether theory of curves $C$ lying on a $K3$ surface $S$ has brought several achievements, such as a new proof of the Gieseker-Petri Theorem \cite{lazarsfeld}, the classification of prime Fano manifolds of coindex $3$ \cite{mukai}, Green's Conjecture for a general curve of any given genus $g$ \cite{voisin1,voisin2}. More recently, there have been applications to higher rank Brill-Noether theory (cf. \cite{angela}) and hyperk\"ahler manifolds (cf. \cite{ciro}). The key observation is that, if $C$ is general in its linear system, non-trivial endomorphisms of the rank $r+1$ LM bundle $E_ {C,(A,V)}$ associated with a linear series $(A,V)\in G^r_d(C)$ measure the failure of injectivity of the Petri map
$$
\mu_{0,(A,V)}:V\otimes H^0(C,\omega_C\otimes A^\vee)\to H^0(C,\omega_C);
$$
if $\rho(g,r,d)<0$ with $g$ equal to the genus of $C$, then the hypothesis on the genericity of $C$ is not necessary. For later use, we recall that in the case of a complete $g^r_d$ on $C$, that is, when $V=H^0(C,A)$, the corresponding LM bundle is denoted by $E_{C,A}$. 

The study of non-simple LM bundles turns out to be crucial. If $r=1$, these arise as extensions of two torsion free sheaves of rank $1$ on $S$ which are the image and the kernel of an endomorphism dropping the rank everywhere (\cite{donagi}). Non-simple LM bundles of rank $3$ (i.e., satisfying $r=2$) were investigated in \cite{ciotola} by using the fact that   they cannot be stable and looking at the corresponding Harder-Narasimhan and Jordan-H\"older filtrations. 

Our idea in order to treat non-simple LM bundles $E_{C,A}$ of arbitrary rank is to consider the pair $(E_{C,A},H^0(S,E_{C,A}))$ as a coherent system à la Le Potier \cite{lepotier}; if $E_{C,A }$ is non-simple, then the same holds for the pair $(E_{C,A},H^0(S,E_{C,A}))$. The notion of stability for coherent systems depends on the choice of a polynomial $\alpha\in \mathbb{Q}[t]$ with positive leading coefficient, and non-simplicity implies the non-stability with respect to every $\alpha$. Having fixed $\alpha$ equal to the Hilbert polynomial $p_S$ of $\oo_S$, we consider the maximal destabilizing sequence of $(E_{C,A},H^0(S,E_{C,A}))$.

A first application of our methods concerns complete linear series $A$ of type $g^r_d$ on a smooth curve $C\subset S$, whenever $A$ computes the Clifford index of $C$ and $\rho(g,r,d)<0$. Set $L:=\oo_S(C)$; in order to state our result, we recall that a line bundle $M\in\Pic(S)$ is \emph{adapted to the linear system $\vert L\vert$} exactly when:
\begin{enumerate}
\item $h^0(S,M)\geq 2$ and $h^0(S,L\otimes M^\vee)\geq 2$;
\item $\Cliff(M\otimes\oo_C)$ is independent of the curve $C\in\vert L\vert$.
\end{enumerate}
Condition (1) makes sure that $M\otimes \oo_C$ contributes to the Clifford index, while (2) is satisfied if either $h^1(S,M)=0$, or $h^1(S,L\otimes M^\vee)=0$.
 
We prove the following:
\begin{thm}\label{mannaggia}
Let $A$ be a complete $g^r_d$ on a non-hyperelliptic and non-trigonal curve $C\subset S$ such that $d\leq g-1$, $\rho(g,r,d)<0$ and $\Cliff(A)=\Cliff(C)$. Assume $L:=\oo_S(C)$ is ample and the following condition is satisfied:
\begin{itemize}
\item[(*)] there is no irreducible elliptic curve $\Sigma\subset S$ such that $\Sigma\cdot C=4$ and no irreducible genus $2$ curve $B\subset S$ such that $B\cdot C=6$.
\end{itemize}
Then, one of the following occurs:
\begin{itemize}
\item[(i)] There exists a line bundle $M\in\Pic(S)$ adapted to $\vert L\vert$ such that $A\simeq M\otimes\oo_C$.
\item[(ii)] The line bundle $A$ satisfies $h^0(C,A)=2$ (i.e., $r=1$); furthermore, there exists a line bundle $M\in\Pic(S)$ adapted to $\vert L\vert$ such that $\vert A\vert$ is contained in the restriction of $\vert M\vert$ to $C$.
\end{itemize}
\end{thm}
In particular, as soon as $d\leq g-1$ and $r>1$, the line bundle $A$ coincides with the restriction to $C$ of a line bundle $M\in\Pic(S)$. In case (ii), the condition that $\vert A\vert$ is contained in the restriction of $\vert M\vert$ to $C$ means that for every $A_0\in\vert A\vert$ there is a divisor $M_0\in \vert M\vert$ such that $A_0\subset M_0\cap C$; if $H^0(S,M)\simeq H^0(C,M\otimes\oo_C)$, this is equivalent to the requirement $h^0(C,A^\vee\otimes M\otimes \oo_C)>0$. 

If condition (*) is violated, there do exist exceptions where neither (i) nor (ii) happen; we actually prove a stronger version of Theorem \ref{mannaggia} (cf. Theorem \ref{duro}), which also covers and completely classifies such exceptional cases. Note that our result makes no assumption on the Clifford index of the curve $C$. As soon as $L:=\oo_S(C)$ is ample, Theorem \ref{mannaggia} can be seen as a refinement of Green and Lazarsfeld's result that all smooth curves in $\vert L\vert$ have the same Clifford index (cf. \cite{green}).

 The exceptional cases mentioned above provide counterexamples to the following conjecture of Donagi and Morrison:
\begin{conj}[\cite{donagi} Conjecture 1.2]\label{falsa}
Let $C$ be a smooth curve of genus $g\geq 2$ lying on a $K3$ surface $S$ and let $A$ be a complete, base point free $g^r_d$ on $C$ such that $d\leq g-1$ and $\rho(g,r,d)<0$. Then, there exists a line bundle $M\in\Pic(S)$ such that $\vert A\vert$ is contained in the restriction of $\vert M\vert$ to $C$ and the following inequalities are satisfied:
$$
c_1(M)\cdot C\leq g-1,\,\,\,\,\,\,\Cliff(M\otimes\oo_C)\leq \Cliff(A).
$$
\end{conj}
\noindent We describe one counterexample explicitly (cf. Counterexample \ref{fortuna}). Let $\pi:S\to \mathbb{P}^2$ be a $2$:$1$ cover branched along a smooth sextic and assume that $B:=\pi^*\oo_{\mathbb{P}^2}(1)$ generates the Picard group of $S$. A curve $C\in\pi^*\vert \oo_{\mathbb{P}^2}(3)\vert\subset \vert 3B\vert$ is a $2$:$1$ cover of an elliptic curve $\Gamma$; let $A$ be the pullback to $C$ of $\oo_\Gamma(P_1+P_2+P_3)$, where $P_1$, $P_2$, $P_3$ are three non-collinear points. One may show that $A$ is a complete $g^2_6$ on $C$ and the linear system $\vert A\vert$ is contained in $\vert \oo_C(2B)\vert$ but not in $\vert\oo_C(B)\vert$; since $2B\cdot C>g-1$, Conjecture \ref{falsa} fails. 

Notice that the above counterexample does not contradict the existence of a line bundle $M\in\Pic(S)$ such that the linear system $\vert A\vert$ is contained in the restriction of $\vert M\vert$ to $C$, but only the inequality $c_1(M)\cdot C\leq g-1$. In fact, one might still believe the following modification of the conjecture:   
\begin{conj}\label{dm}
Let $C$ be a smooth irreducible curve of genus $g\geq 2$ lying on a $K3$ surface $S$ and $A$ be a complete, base point free $g^r_d$ on $C$ such that $d\leq g-1$ and $\rho(g,r,d)<0$. Then, there exists a line bundle $M\in\Pic(S)$, adapted to $\vert L\vert$, such that:
 \begin{itemize}
 \item[\em (i)] $\vert A\vert$ is contained in the restriction of $\vert M\vert$ to $C$;
 \item[\em (ii)] $\Cliff(M\otimes\oo_C)\leq \Cliff(A)$. 
 \end{itemize}
\end{conj}
The condition $c_1(M)\cdot C\leq g-1$ is here replaced with the requirement that $M$ is adapted to $\vert L\vert$. Theorem (5.1') in \cite{donagi} proves both Conjecture \ref{falsa} and Conjecture \ref{dm} for $r=1$. The refined Conjecture \ref{dm} for $r=2$ was obtained in \cite[Theorem 1.1]{ciotola} under some mild hypotheses on the line bundle $L$; note that, for $r=2$, the inequality $c_1(M)\cdot C\leq (4g-4)/3$ was also obtained and Counterexample \ref{fortuna} shows that this bound is optimal.

By using coherent systems and a generalization of the notion of LM bundles, we reduce Conjecture \ref{dm} to a question of secant varieties. Our main result is the following:
\begin{thm}\label{chiana}
Let $S$ be a smooth projective $K3$ surface containing no ($-2$)-curves. Let $A$ be a complete, base point free $g^r_d$ on a smooth genus $g$ curve $C\subset S$ such that $d\leq g-1$ and $\rho(g,r,d)<0$, and assume that the pair $(C,A)$ has no unexpected secant varieties up to deformation. Then, Conjecture \ref{dm} holds for $A$.
\end{thm}
We explain the statement of the theorem by spending a few words on the new concept of {\em having no unexpected secant varieties up to deformation}.

Given $A\in W^r_d(C)$ and having fixed integers $0\leq f<e$, the {\em variety of secant divisors $V^{e-f}_e(A)$} is the determinantal variety:
$$V^{e-f}_e(A):=\{D\in C_e\,\vert\, h^0(C,A(-D))\geq r+1-e+f\}.$$
By definition, the variety $V^{e-f}_e(A)$ parametrizes effective divisors of degree $e$ on $C$ which impose at most $e-f$ conditions on the linear system $\vert A\vert$. If $A$ is very ample, these are the ($e-f-1$)-planes which are $e$-secant to the embedded curve $C\stackrel{\vert A\vert}{\hookrightarrow} \mathbb{P}^r$. Also note that, when $A=\omega_C$, the variety $V^{e-f}_e(A)$ is the inverse image of the Brill-Noether variety $W^f_e(C)$ under the Abel-Jacobi map $C_e\to \Pic(C)$. It is classically known that:
 $$\mathrm{expdim}\,V^{e-f}_e(A)=e-f(r+1-e+f);$$
 we refer to \cite{gabi} for results concerning the existence (resp. non-existence) of linear series with special secancy conditions on an arbitrary (resp. general) genus $g$ curve.

 Replace now $C$ with an integral (and possibly singular) curve $X$, and $A$ with a torsion free sheaf $B\in \overline{W}^r_d(X)$, i.e., lying in the compactified Jacobian $\overline{J}^d(X)$ and satisfying $h^0(X,B)\geq r+1$. Then, the definition of the secant variety $V^{e-f}_e(B)$ still makes sense with some slight modifications:
 $$
 V^{e-f}_e(B):=\{q\in \mathrm{Quot}_{B}^e\,\vert\, h^0(X,\ker q)\geq r+1-e+f\},
 $$
 where the Quot scheme $\mathrm{Quot}_{B}^e$ parametrizes quotients $q:B\to Q$ of degree $e$.

Given $C$ and $A$ as in Conjecture \ref{dm}, we say that the pair $(C,A)$ {\em has some unexpected secant varieties up to deformation} if it can be deformed to a pair $(X,B)$ such that the following hold:
 \begin{itemize}
 \item The curve $X\in\vert L\vert$ is integral, the sheaf $B\in\overline{J}^d(X)$ is globally generated and $E_{X,B}\simeq E_{C,A}$; in particular, $h^0(X,B)=r+1$.
 \item For some integers $0\leq f<e$, one has $V^{e-f}_e(B)\neq \emptyset$ and $\mathrm{expdim}\,V^{e-f}_e(B)<0$.
 \end{itemize}
 
 The bundle $E_{X,B}$ above is defined in the same way as LM bundles for line bundles on smooth curves, that is, $E_{X,B}$ is the dual of the kernel of the evaluation map $ H^0(X,B)\otimes \oo_S\twoheadrightarrow B$.

In the joint Appendix with Andreas Leopold Knutsen, we show that the hypothesis in Theorem \ref{chiana} concerning secant varieties cannot be avoided by exhibiting a counterexample to Conjecture \ref{dm}. The failure of the conjecture is obtained along with the existence of some unexpected secant varieties on a deformation of the pair $(C,A)$.

The organization of the paper is as follows. In Section 2, we introduce generalized Lazarsfeld-Mukai bundles (g.LM bundles in the sequel, cf. Definition \ref{inizio}), which share some common properties with LM bundles. In particular, the definition of the Clifford index of a g.LM bundle $E$ will play a central role in the proof of Theorems \ref{mannaggia} and \ref{chiana}; by Corollary \ref{index}, $\Cliff(E)$ is nonnegative as soon as $c_1(E)^2>0$. The case $c_1(E)^2=0$ is covered by Proposition \ref{elliptic}.

Section 3 contains some preliminaries on coherent systems. Theorem \ref{riso} expresses any non-simple and globally generated LM bundle as an extension of a g.LM bundle of type (II) by the elementary modification of a g.LM bundle of type (I). 
 
In Section 4, Theorem \ref{mannaggia} is proved. The proof is made technically difficult by the possible presence of smooth curves on $S$ which are rational or elliptic or hyperelliptic.

Finally, in Section 5 we focus on Conjectures \ref{falsa} and \ref{dm} and prove Theorem \ref{chiana}. The strategy consists in considering the maximal destabilizing pair $(E_1,H^0(E_1))$ of the coherent system $(E_{C,A},H^0(E_{C,A}))$ and in showing that, as soon as the vector bundle $E_1$ has rank $\geq 2$, some deformation of the pair $(C,A)$ has an unexpected secant variety.\\\vspace{0.2cm}

\textbf{Acknowledgements:} I have benefited from interesting conversations and correspondence with Ciro Ciliberto, Gavril Farkas, Daniel Huybrechts, Andreas Leopold Knutsen and Peter Newstead, and I would like to thank all of them. I especially thank Alessandro D'An\-drea, Gavril Farkas and Andreas Leopold Knutsen for useful comments concerning a preliminary version of the paper.  This work was done during my stay at the Max Planck Institute for Mathematics in Bonn and I am grateful to this institution for the warm hospitality!
 
\subsection{Notation and preliminaries}
For us, $S$ will always be a smooth projective $K3$ surface and $C$ a smooth irreducible curve on it, whose genus is denoted by $g$.  A linear series $(A,V)$ of type $g^r_d$ on $C$ is called {\em primitive} if both $(A,V)$ and $\omega_C\otimes A^\vee$ are base point free. We set $L:=\oo_S(C)$, and we denote by $\W^r_d(\vert L\vert)$ the variety parametrizing pairs $(C',A')$ with $C'$ a smooth irreducible curve in the linear system $\vert L\vert$ and $A'\in W^r_d(C')$.

In all cases where no confusion arises, given a sheaf $\mathcal{F}$ on a scheme $Y$, we will simplify notation and write $H^i(\mathcal{F})$ (or $h^i(\mathcal{F})$ for the corresponding dimension), dropping any reference to $Y$.

Given a $0$-dimensional sheaf $\tau$ on $S$, we denote by $l(\tau)$ its length, which coincides by definition with $h^0(\tau)$.

All diagrams appearing in the paper are commutative and all of their columns and rows are exact.

Throughout the paper, we will make frequent use of the following strong version of Bertini's Theorem, due to Saint-Donat.
\begin{thm}[\cite{donat}]\label{bertini}
Let $L$ be a line bundle on a $K3$ surface $S$ such that $h^0(S,L)>0$. Then, $\vert L\vert$ has no base points outside its fixed components. Furthermore, if $\vert L\vert$ has no base components, then either of the following holds:
\begin{enumerate}
\item[\em (i)] If $c_1(L)^2>0$, then $h^1(S,L)=0$ and a general element in $\vert L\vert$ is a smooth, irreducible curve of genus $g=1+c_1(L)^2/2$.
\item[\em (ii)] If $c_1(L)^2=0$, then there exist a number $k\in\mathbb{Z}^{>0}$ and an irreducible curve $\Sigma\subset S$ with $p_a(\Sigma)=1$ such that $L=\oo_S(k\Sigma)$. In this case, one has $h^0(S,L)=k+1$, $h^1(S,L)=k-1$ and every  element in $\vert L\vert$ can be written as a sum $\Sigma_1+\Sigma_2+\cdots+\Sigma_k$ with $\Sigma_i\in\vert \Sigma\vert$ for $1\leq i\leq k$.
\end{enumerate}
\end{thm}
We recall that an effective divisor $D\subset S$ is called {\em numerically $m$-connected} if, whenever $D= D_1+D_2$ with $D_1$ and $D_2$ effective and nonzero, one has $D_1\cdot D_2\geq m$. The following result is used in Sections 4 and 5.
\begin{thm}[\cite{donat, reid}]
Let $S$ and $L$ be as in Theorem \ref{bertini} and assume $c_1(L)^2>0$. Then, $\vert L\vert$ has no fixed components if and only if every divisor in $\vert L\vert$ is numerically $2$-connected.
\end{thm}

\section{generalized Lazarsfeld-Mukai bundles}
We start by recalling the definition of LM bundles. We refer to \cite{aprodu} for an exhaustive survey about the topic. If $(A,V)$ is a base point free $g^r_d$ on a smooth irreducible curve $C\subset S$, the LM bundle $E_{C,(A,V)}$ is by definition the dual of the kernel of the evaluation map $V\otimes \oo_S\twoheadrightarrow A$. In particular, $E_{C,(A,V)}$ sits in the short exact sequence
$$
0\to V^\vee\otimes \oo_S\to E_{C,(A,V)}\to \omega_C\otimes A^\vee\to 0,
$$
and $V^\vee$ defines an ($r+1$)-dimensional subspace of the global sections of $E_{C,(A,V)}$. When the $g^r_d$ on $C$ is complete,  then the LM bundle is simply denoted by $E_{C,A}$. Here are some basic facts concerning LM bundles:
\begin{prop}
The LM bundles $E_{C,(A,V)}$ satisfies the following properties:
\begin{itemize}
\item $\rk\,E_{C,(A,V)}=r+1$, $c_1(E_{C,(A,V)})=c_1(L)$, $c_2(E_{C,(A,V)})=d$;
\item $\chi(E_{C,(A,V)})=g-d+2r+1$, $h^1(E_{C,(A,V)})=h^0(A)-r-1$, $h^2(E_{C,(A,V)})=0$;
\item $E_{C,(A,V)}$ is globally generated off the base locus of $\omega_C\otimes A^\vee$;
\item $\chi(E_{C,(A,V)}^\vee\otimes E_{C,(A,V)})=2(1-\rho(g,r,d))$, and hence $E_{C,(A,V)}$ is non-simple if $\rho(g,r,d)<0$;
\item if $C\in \vert L\vert$ is general and $\mu_{0,(A,V)}:V\otimes H^0(\omega_C\otimes A^\vee)\to H^0(\omega_C)$ is the Petri map, then the simplicity of $E_{C,(A,V)}$ is equivalent to the injectivity of $\mu_{0,(A,V)}$.
\end{itemize}
\end{prop}
In the sequel, we will especially treat the cases where $\omega_C\otimes A^\vee$ is base point free. We will show that every non-simple globally generated LM bundle is the extension of two sheaves of a special type. With this in mind, we introduce the following:
\begin{df}\label{inizio}
A torsion free sheaf $E\in\mathrm{Coh}(S)$ is called a \emph{generalized Lazarsfeld-Mukai bundle} (g.LM bundle in the sequel) iff $h^2(E)=0$ and either
\begin{enumerate}
\item[\em (I)] $E$ is locally free and generated by global sections off a finite set;\\
or
\item[\em (II)] $E$ is globally generated.
\end{enumerate}
\end{df}
\begin{rem}
If conditions (I) and (II) of the above definition are both satisfied, then $E$ is the LM bundle associated with a smooth irreducible curve $C\subset S$ and a primitive linear series $(A,V)$ on $C$ (i.e., $E=E_{C,(A,V)}$). Furthermore, $V=H^0(A)$ if and only if $h^1(E)=0$.
\end{rem}
Definition \ref{inizio} is motivated  from the fact that g.LM bundles have properties similar to LM bundles. This is proved in what follows.
\begin{prop}\label{general}
Let $E$ be a g.LM bundle of type (I) and rank $r_E$, and $D_{t-1}(V)$ denote the degeneracy locus of the evaluation map $ev_V:V\otimes \oo_S\to E$ for a general subspace \mbox{$V\in G(t,H^0(E))$}. Then, the following are satisfied:
\begin{itemize}
\item[\em (i)] if $t\leq r_E-2$, then $D_{t-1}(V)$ is empty;
\item[\em (ii)] for $t=r_E-1$, the locus $D_{r_E-2}(V)$ consists of a finite number of distinct points, none of which lies in the locus where $E$ is not globally generated;
\item[\em (iii)] for $t=r_E$, the locus $D_{r_E-1}(V)$ is $1$-dimensional. If $c_1(E)^2>0$, then $D_{r_E-1}(V)$ is an integral curve $X$ (possibly singular at the points at which $E$ is not globally generated) and the cokernel of the evaluation map $ev_V$ is a torsion free sheaf of rank $1$ on $X$.
\end{itemize}

Furthermore, given any closed subset $K$ of the total space of $E$ such that $\dim\,K=r_E-t+1$, the image of $ev_V$ meets $K$ in a (possibly empty) finite set.
\end{prop}
\begin{proof}
We look at the short exact sequence
\begin{equation}\label{mo}
0\to\tilde{E}\to E\to \tau\to 0,
\end{equation}
where $\tilde{E}$ is a globally generated subsheaf of $E$ satisfying $H^0(\tilde{E})\simeq H^0(E)$, and $\tau$ is a $0$-dimensional sheaf supported on the points at which $E$ is not globally generated. Since $h^1(\tilde{E})=h^1(E)+l(\tau)$, there exists an exact sequence
\begin{equation}\label{ma}
0\to V_1\otimes \oo_S\stackrel{ev_1}{\longrightarrow} \overline{E}\stackrel{p}{\longrightarrow} E\to\tau\to 0,
\end{equation}
where $\dim\,V_1=l(\tau)$ and the cokernel of $ev_1$ coincides with $\tilde{E}$. It can be easily shown that $\overline{E}$ is a globally generated vector bundle on $S$ satisfying $h^0(\overline{E})=h^0(E)+l(\tau)$ and $h^i(\overline{E})=h^i(E)$ for $i=1,2$ (cf. \cite[Lemma 1.6]{green}). 

Take a general $V_2\in G(t, H^0(\overline{E}))$. By a general position argument (cf. \cite[B.9.1]{fulton}), if $t\leq r_E-2$ (respectively $t=r_E-1$) one can choose $V_2$ such that the degeneracy locus of the evaluation map $ev=(ev_1,ev_2):(V_1\oplus V_2)\otimes\oo_S\to \overline{E}$ coincides with that of $ev_1$ (resp. is $0$-dimensional and for all $x\in\mathrm{supp}(\tau)$ the kernel of $ev_x$ coincides with that of $(ev_1)_x$). Thus, $(i)$ and $(ii)$ follow because the map induced by $p$ on the global sections sends a general $V_2\in G(t, H^0(\overline{E}))$ to a general point of the Grassmannian $G(t,H^0(E))$.

Given a general subspace $V_2\in G(r_E-1, H^0(\overline{E}))$, we consider the following commutative diagram:
$$
\xymatrix{
&&0&0\\
0\ar[r]&V_1\otimes \oo_S\ar[r]_{ev_1'}&\overline{E}'\ar[u]\ar[r]&\det E\otimes I_\xi\ar[u]\ar[r]&\tau\ar[r]&0\\
0\ar[r]&V_1\otimes \oo_S\ar@{=}[u]\ar[r]_{ev_1}&\overline{E}\ar[u]\ar[r]&E\ar[u]\ar[r]&\tau\ar@{=}[u]\ar[r]&0.\\
&&V_2\otimes\oo_S\ar[u]\ar@{=}[r]&V_2\otimes\oo_S\ar[u]\\
&&0\ar[u]&0\ar[u]\\
}
$$
The sheaf $\overline{E}'$ is locally free and $\xi$ is a $0$-dimensional subscheme of $S$ disjoint from $\mathrm{supp}(\tau)$. The cokernel $\tilde{E}'$ of $ev_1'$ is a globally generated, torsion free sheaf of rank $1$ on $S$. The vanishing locus $X$ of a general section $\alpha'\in H^0(\tilde{E}')$ is $1$-dimensional; if the inequality $c_1(E)^2>0$ is satisfied, then $X$ is an integral curve, which is smooth whenever $\tau$ is curvilinear. One can lift $\alpha'$ to a section ${\alpha}\in H^0(\overline{E})$. Having denoted by $V$ the image of $\langle V_2,\alpha\rangle$ in $H^0(E)$, the evaluation map $ev_V$ is injective and its cokernel is a pure sheaf of dimension $1$ supported on $X$. Hence, item $(iii)$ follows.  

Concerning the last part of the statement, one has $\dim\,p^{-1}(K)\leq \dim\,K+l(\tau)$. If $V_2\in G(t,H^0(\overline{E}))$ is general, then \cite[B.9.1]{fulton} implies that the image of the evaluation map $ev_2:V_2\otimes\oo_S\to \overline{E}$ meets $p^{-1}(K)$ in at most a finite number of points; this concludes the proof.
\end{proof}
As regards g.LM bundles of type (II), the following analogue of Proposition \ref{general}(iii) holds; we only give a sketch of the proof since we do not use it in the rest of the paper.
\begin{prop}
Let $E$ be a g.LM bundle of type (II) and rank $r_E$, and let $V\in G(r_E,H^0(E))$ be general. Then, the evaluation map $ev_V:V\otimes \oo_S\to E$ is injective and its cokernel is  a pure sheaf $B$ of dimension $1$ on $S$. If moreover $c_1(E)^2>0$, then $B$ is a torsion free sheaf of rank $1$ on an integral curve $X\subset S$.
\end{prop}
 \begin{proof}
We consider the short exact sequence 
 \begin{equation}\label{ancora}
 0\to E\to E^{\vee\vee}\to\kappa\to 0,
 \end{equation}
 where $\kappa$ is a $0$-dimensional sheaf on $S$. It is enough to observe that $H^0(E)\subset H^0(E^{\vee\vee})$ generates $E^{\vee\vee}$ off the support of $\kappa$ and to apply Bertini's Theorem. \end{proof}
 Our next goal is to find a lower bound for the second Chern class of a g.LM bundles of given rank. It is convenient to introduce the following:
\begin{df}
Let $E$ be a g.LM bundle. The {\em Clifford index of $E$} is: $$\Cliff(E):=c_2(E)-2(\rk\,E-1).$$
\end{df}
\begin{rem}
If $E=E_{C,A}$ for a smooth irreducible curve $C\subset S$ and $A\in\Pic(C)$, one has the equality $\Cliff(E_{C,A})=\Cliff(A)$.
\end{rem}
Starting from this observation, we prove the following:
\begin{prop}\label{cliff}
Let $E$ be a g.LM bundle such that $c_1(E)^2>0$. If $E$ is of type (I), then the following inequality is satisfied:
$$
\Cliff(E)\geq 2h^1(E)+l(\tau),
$$
where $\tau$ is the $0$-dimensional sheaf appearing in the exact sequence (\ref{mo}). If instead $E$ is of type (II), we have:
$$
\Cliff(E)\geq \Cliff(E^{\vee\vee})+l(\kappa),
$$
where $\kappa$ is the $0$-dimensional sheaf appearing in the exact sequence (\ref{ancora}).
\end{prop}
\begin{proof}
First of all, assume $E$ satisfies both conditions (I) and (II) in Definition \ref{inizio}. Then, there exist a smooth irreducible curve $C\subset S$ and $A\in\Pic(C)$ such that $E_{C,A}$ sits in a short exact sequence:
\begin{equation}
0\to\oo_S^{\oplus h^1(E)}\to E_{C,A}\to E\to0.
\end{equation}
It turns out that
\begin{equation}\label{uno}
\Cliff(E)=\Cliff(E_{C,A})+2h^1(E)\geq 2h^1(E),
\end{equation}
where the last inequality follows from Clifford's Theorem. 

Now, we consider the case where $E$ is of type (I). The globally generated vector bundle $\overline{E}$ appearing in the exact sequence (\ref{ma}) satisfies $h^1(\overline{E})=h^1(E)$, and $c_2(\overline{E})=c_2(E)+l(\tau)$, and $\rk\,\overline{E}=\rk\,E+l(\tau)$. Inequality (\ref{uno}) for $\overline{E}$ yields:
\begin{equation}\label{due}
\Cliff(E)=\Cliff(\overline{E})+l(\tau)\geq 2h^1(E)+l(\tau).
\end{equation}

In order to cover the case of a g.LM bundle $E$ of type (II), it is enough to remark that $\rk\,E^{\vee\vee}=\rk\,E$ and $c_2(E^{\vee\vee})=c_2(E)-l(\kappa)$.
\end{proof}
\begin{cor}\label{index}
Let $E$ be a g.LM bundle of rank $r_E$ and $c_1(E)^2>0$. Then, $\Cliff(E)\geq 0$ and equality holds only in the following cases:
\begin{itemize}
\item[(a)] $r_E=1$ and $E$ is a globally generated line bundle;
\item[(b)] $E=E_{C,\omega_C}$ for some smooth irreducible curve $C\subset S$ of genus $g=r_E\geq 2$;
\item[(c)] $r_E>1$ and $E=E_{C,(r_E-1)g^1_2}$ for some smooth hyperelliptic curve $C\subset S$ of genus $g>r_E$.
\end{itemize}
\end{cor}
\begin{proof}
Proposition \ref{cliff} trivially implies the first part of the statement. If $\Cliff(E)=0$, then $E$ is both locally free and globally generated and satisfies $h^1(E)=0$. Hence, $E$ is the LM bundle associated with a smooth irreducible curve $C\subset S$ and a primitive line bundle $A\in\Pic(C)$ such that $\Cliff(A)=0$. Case (a) occurs when $A$ is the structure sheaf of $C$. If instead $\mathrm{deg}(A)>0$, then Clifford's Theorem implies that $A$ is either the canonical sheaf $\omega_C$ (case (b)), or a multiple of a linear series of type $g^1_2$ on $C$ (case (c)). In the latter case, the inequality $g>r_E$ follows from the fact that the residual of the linear series of type $(r_E-1)g^1_2$ is globally generated (it is a quotient of $E$) and by imposing $(r_E-1)g^1_2\neq\omega_C$.
\end{proof}
With regard to item (c) of Corollary \ref{index}, we recall the classification of hyperelliptic linear systems on $S$ due to Saint-Donat (\cite[Theorem 5.2]{donat}).
\begin{thm}
Let $C\subset S$ be a smooth hyperelliptic curve of genus $g\geq 2$ and set $L:=\oo_S(C)$. Then, one of the following occurs:
\begin{itemize}
\item The equality $c_1(L)^2=2$ holds.
\item There is a smooth, irreducible curve $B\subset S$ of genus $2$ satisfying $L\simeq \oo_S(2B)$.
\item There exists an irreducible elliptic curve $\Sigma\subset S$ such that $c_1(L)\cdot\Sigma=2$.
\end{itemize}
\end{thm}
Heretofore, we have treated g.LM bundles $E$ satisfying $c_1(E)^2>0$. The following result concerns g.LM bundles whose first Chern class has vanishing self-intersection.
\begin{prop}\label{elliptic}
Let $E$ be a g.LM bundle such that $c_1(E)^2=0$. Then, $E$ is both locally free and globally generated and satisfies $c_2(E)=0$. Furthermore, if $h^1(E)=0$, then \mbox{$E=\oo_S(\Sigma)^{\oplus\rk E}$} for an irreducible elliptic curve $\Sigma\subset S$.
\end{prop}
\begin{proof}
Let $E$ be of type (I). Up to replacing $E$ with the bundle $E'$ sitting in the short exact sequence
\begin{equation}\label{treno}
0\to \oo_S^{\oplus h^1(E)}\to E'\to E\to 0,
\end{equation}
we can assume $h^1(E)=0$. Then, Proposition (1.5) in \cite{green} yields $E=\oo_S(\Sigma)^{\oplus \rk E}$ for an irreducible elliptic curve $\Sigma\subset S$ and our statement follows trivially.

Now, let $E$ be of type (II) and consider the short exact sequence
$$
0\to E\to E^{\vee\vee}\to\kappa\to 0.
$$

Since $E^{\vee\vee}$ is a g.LM bundle of type (I), it satisfies $c_2(E^{\vee\vee})=0$. Therefore,  in order to prove that $\kappa=0$, it is enough to show that $c_2(E)=0$, too. To this purpose we may assume $h^1(E)=0$. We define $D(E)$ to be the dual of the kernel of the evaluation map $H^0(E)\otimes\oo_S\twoheadrightarrow E$, and we obtain the following exact sequence:
$$
0\to E^\vee\to H^0(E)^\vee\otimes\oo_S\to D(E)\to\mathcal{E}xt^1(E,\oo_S)\to 0.
$$

The condition $h^1(E)=0$ yields the isomorphism $H^0(D(E))\simeq H^0(E)^\vee$. As a consequence, $D(E)$ is a LM bundle of type (I) which is globally generated off the support of $\mathcal{E}xt^1(E,\oo_S)\simeq  \mathcal{E}xt^2(\kappa,\oo_S)$. Since $c_1(D(E))^2=0$, we conclude that $\mathcal{E}xt^2(\kappa,\oo_S)=0$ and $\kappa=0$ as well.

\end{proof}

\begin{rem}\label{postino}
Proposition \ref{elliptic} trivially implies that, if $E$ is a g.LM bundle satisfying $c_1(E)^2=0$, then $\Cliff(E)=-2(\rk E-1)$.
\end{rem}

\section{Coherent systems}
We give some background information about coherent systems and refer to \cite{lepotier,he} for details.
\begin{df}
A coherent system of dimension $d$ on a smooth projective variety $Y$ is a pair $(E, V)$, where $E\in\mathrm{Coh}(Y)$ is a $d$-dimensional sheaf and $V$is a vector subspace of $H^0(E)$. A morphism of coherent systems $f:(E,V)\to (E',V')$ is a morphism of coherent sheaves $f:E\to E'$ such that $f(V)\subset V'$.
\end{df}
Given a coherent system $\Lambda=(E,V)$, we denote by $\Ext^i(\Lambda,-)$ the derived functors of the functor $\Hom(\Lambda,-)$. We recall the following result:
\begin{prop}[\cite{he}, Corollaire 1.6]\label{he}
Given two coherent systems $\Lambda=(E,V)$ and $\Lambda'=(E',V')$ on $Y$, there exists an exact sequence:
\begin{eqnarray*}
0\to\Hom(\Lambda,\Lambda')&\to&\Hom(E,E')\to\Hom(V,H^0(E')/V')\\
\to\Ext^1(\Lambda,\Lambda')&\to&\Ext^1(E,E')\to\Hom(V,H^1(E'))\\
\to\Ext^2(\Lambda,\Lambda')&\to&\Ext^2(E,E')\to\Hom(V,H^2(E'))\to\cdots.
\end{eqnarray*}
\end{prop}
In particular, if $E$ is non-simple, then the coherent system $(E,H^0(E))$ is non-simple as well.

One can introduce a notion of stability for coherent systems depending on the choice of a polynomial $\alpha\in \mathbb{Q}[t]$ with positive leading coefficient. Having fixed $\alpha$, we define a polynomial $p^\alpha_{(E,V)}\in\mathbb{Q}[t]$ by setting
$$
p^\alpha_{(E,V)}(t):=\frac{\dim\,V}{r_E}\alpha(t)+p_E(t),
$$
where $p_E$ is the reduced Hilbert polynomial of $E$ and $r_E$ is its multiplicity. We say that $(E,V)$ is $\alpha$-semistable if $E$ is pure and any coherent subsheaf $F\subset E$ satisfies:
\begin{equation}\label{cs}
p^\alpha_{(F,V')}\leq p^\alpha_{(E,V)},\textrm{     with  }V':=V\cap H^0(F).
\end{equation}
The $\alpha$-stability for $(E,V)$ requires that inequality (\ref{cs}) be always strict and implies simplicity of $(E,V)$.

We specialise to the case of a smooth projective $K3$ surface $S$ and fix $\alpha$ equal to the Hilbert polynomial $p_S$ of $\oo_S$; for convenience, we will talk about (semi)stability when referring to $p_S$-(semi)stability. With this choice, inequality (\ref{cs}) becomes equivalent to the requirement:
\begin{equation}
\frac{\dim\,V'}{r_F}\leq\frac{\dim\,V}{r_E}\,\,\,\textrm{    and, if "=" holds, then }p_F(t)\leq p_E(t).
\end{equation}
\begin{rem}\label{facile}
If $(E,V)$ is a semistable coherent system of dimension $d$, then the evaluation map $ev_V:V\otimes\oo_S\to E$ is generically surjective. In particular, if $E$ has no torsion (i.e., $d=2$), we obtain that $\rk(E)\leq \dim\,V$. 
\end{rem}
Given an unstable coherent system $(E,V)$ such that $E$ is torsion free, we look at its Harder-Narasimhan filtration
$$
0\subset (HN_1(E),V_1)\subset\ldots\subset (HN_{s-1}(E),V_{s-1})\subset (HN_{s}(E),V_{s})=(E,V).
$$
We use the following:
\begin{lem}\label{pizza}
Assume that $E$ is globally generated by the sections in $V$ and $h^2(E)=0$. Then for all $i$ the following are satisfied:
\begin{enumerate}
\item $V_i=V\cap H^0(HN_i(E))$.
\item The sheaf $HN_i(E)$ is generically generated by the sections in $V_i$.
\item For any subsheaf $F\subseteq HN_i(E)$ with $H^0(F)\cap V\simeq V_i$, one has $h^2(F)=0$.
\end{enumerate}
The same holds true if $(E,V)$ is semistable and we replace the Harder-Narasimhan filtration with the Jordan-H\"older one.
\end{lem}
\begin{proof}
First of all, note that the vanishing $h^2(E)=0$ implies $\dim V>\rk\,E$; indeed, if we had $\dim V=\rk\,E$, then $E$ would be isomorphic to the direct sum of $\rk\,E$ copies of $\oo_S$ and $h^2(E)>0$. 

Since for all $i$ the pair $(E/HN_i(E), V/V_i)$ satisfies the same properties as $(E,V)$, by induction on the length of the filtration it is enough to prove the lemma for the maximal destabilizing coherent system $(HN_1(E),V_1)$, which is semistable. Then, the first two points trivially follow from maximality and Remark \ref{facile}. 

Let $F$ be as in point (3); then, one has $\rk\,F=\rk\, HN_1(E)<\dim\,V_1$ and the sheaf $T$ appearing in the short exact sequence
$$
0\to F\to HN_1(E)\to T\to 0
$$
is a torsion sheaf. By Serre duality, one has to show that $\Hom(F,\oo_S)=0$. Assume there exists a non-zero morphism $\psi:F\to \oo_S$ and denote its kernel by $K$. Since $\im\,\psi$ is torsion free and generically generated by global sections, then $\psi$ is surjective and the pair $(K,H^0(K)\cap V_1)$ destabilizes $(HN_1(E),V_1)$, thus yielding a contradiction.
\end{proof}

\begin{prop}\label{generated}
Let $E$ be a vector bundle on $S$ generically generated by its global sections. Assume moreover that every subsheaf $F\subseteq E$ for which $H^0(F)\simeq H^0(E)$ also satisfies $H^2(F)=0$.\\ Then, either the $1$-dimensional locus where $E$ is not generated by its global sections is a (possibly empty) union of ($-2$)-curves, or there exists a nef line bundle $N$ along with an injective map $N\hookrightarrow E$.
\end{prop}
\begin{proof}
There exists a short exact sequence
\begin{equation}\label{gg}
0\to \tilde{E}\to E\to T\to 0,
\end{equation}
where $T$ is a pure sheaf of dimension $1$ supported on a (possibly singular and reducible) curve $\gamma$, and the sheaf $\tilde{E}$ satisfies $H^0(\tilde{E})\simeq H^0(E)$ and is globally generated off a finite set. Since we are on a surface, the sheaf $T$ is reflexive (cf. \cite[Proposition 1.1.10]{lehn}), i.e., $T\simeq T^{\vee\vee}$ with $T^\vee:=\mathcal{E}xt^1(T,\oo_S)$. From (\ref{gg}) one gets that $\mathcal{E}xt^i(\tilde{E}, \oo_S)=0$ for $i=1,2$, hence $\tilde{E}$ is locally free. By assumption, $h^2(\tilde{E})=0$; therefore, $\tilde{E}$ is a g.LM bundle of type (I). Let $\gamma=\gamma_1\cup\cdots\cup\gamma_s$ be the decomposition of $\gamma$ in irreducible components and let $r_i$ denote the generic rank of $T$ along $\gamma_i$.

First of all, we reduce to the case $s=1$. Define $\gamma^i:=\gamma_1\cup\ldots\cup\hat{\gamma_i}\cup\ldots\gamma_s$ for all $1\leq i\leq s$. The restriction of $T$ to $\gamma^i$ may have some torsion $T_0(T\vert_{\gamma^i})$ (cf. \cite[3.1]{giulia} for a detailed discussion); set $T^{(i)}:=T\vert_{\gamma^i}/T_0(T\vert_{\gamma^i})$ and denote by $T_{(i)}$ the kernel of the surjective map $T\to T^{(i)}$, which is supported on $\gamma_i$. Consider the following commutative diagram:
$$
\xymatrix{
&0\\
&T_{(i)}\ar[u]&&0\\
0\ar[r]&E^{(i)}\ar[r]\ar[u]&E\ar[r]&T^{(i)}\ar[r]\ar[u]&0\\
0\ar[r]&\tilde{E}\ar[r]\ar[u]&E\ar[r]\ar@{=}[u]&T\ar[r]\ar[u]&0.\\
&0\ar[u]&&T_{(i)}\ar[u]\\
&&&0\ar[u]
}
$$
Since $T^{(i)}$ is pure, then $E^{(i)}$ is locally free and $T_{(i)}$ is pure as well. The $1$-dimensional locus at which $E^{(i)}$ is not globally generated coincides with $\gamma_i$ and $H^0(E^{(i)})\simeq H^0(E)$. Therefore, it is enough to treat the case where $\gamma$ is irreducible and provide an injective morphism $\oo_S(\gamma)\hookrightarrow E$.

Let $s=1$. The restriction of (\ref{gg}) to $\gamma=\gamma_1$ gives an exact sequence:
$$
0\to K\to \tilde{E}\vert_\gamma\to E\vert_\gamma\to T\to 0.$$
We look at $K$ as a subset of the total space of $\tilde{E}$ with $\dim\,K=r_1+1$. By Proposition \ref{general}, given a general subspace $V\in G(\rk\, E-r_1,H^0(\tilde{E}))$, the image of the evaluation map $ev_{V}$ meets $K$ on a reduced $0$-dimensional subscheme $\zeta\subset \gamma$, that we may assume disjoint from the singular locus of $\gamma$. We look at the diagram:
 $$
 \xymatrix{
&0&0\\
0\ar[r]&\tilde{E'}\ar[u]\ar[r]&E'\ar[u]\ar[r]^\alpha&T\ar[r]&0\\
0\ar[r]&\tilde{E}\ar[u]\ar[r]&E\ar[u]\ar[r]&T\ar@{=}[u]\ar[r]&0.\\
&V\otimes\oo_S\ar[u]\ar@{=}[r]&V\otimes\oo_S\ar[u]\\
&0\ar[u]&0\ar[u]\\
}
$$
If $r_1=1$, one can choose $V$ such that $\tilde{E'}=\det\tilde{E}\otimes I_\xi$ and $E'=\det E\otimes I_\xi\otimes I_\zeta$, where $\xi$ is a $0$-dimensional subscheme of $S$ disjoint from $\gamma$. If instead $r_1>1$, then $\tilde{E}'$ is locally free while $E'$ is torsion free and satisfies $\mathrm{Sing}E'=\zeta$. The map $\alpha$ factors through a surjective map $\alpha':E'\vert_\gamma\to T$. Since $E'\vert_\gamma$ has torsion along $\zeta$ and has the same rank as $T$, we obtain $\det \tilde{E}=\det E(-\gamma)$ if $r_1=1$, and $\tilde{E}'=E'(-\gamma)^{\vee\vee}$ if $r_1>1$. The isomorphism $H^0(\tilde{E}')\simeq H^0(E')$ implies that every section of $E'$ vanishes along $\gamma$ and thus  factors through a map $\oo_S(\gamma)\stackrel{\iota}{\hookrightarrow}  E'$. The irreducibility of $\gamma$ yields that either $\gamma$ is a ($-2$)-curve, or $\oo_S(\gamma)$ is nef. In both cases the vanishing $h^1(\oo_S(\gamma))=0$, given by the strong version of Bertini's Theorem, enables us to lift $\iota$ to a morphism $\oo_S(\gamma)\hookrightarrow E$. This concludes the proof.

\end{proof}

If $E$ is a vector bundle as in Proposition \ref{generated}, we say that $E$ is an {\em elementary modification of a g.LM bundle of type (I)}. As a corollary, we obtain the following:
\begin{thm}\label{riso}
Let $E_{C,A}$ be a non-simple LM bundle on $S$ associated with a primitive line bundle $A$ on $C$. Then, $E$ is the extension of a g.LM bundle $E_2$ of type (II) by a vector bundle $E_1$, which is an elementary modification of a g.LM bundle of type (I); in particular, one has $h^1(E_2)=h^2(E_1)=0$.
\end{thm}
\begin{proof}
Since $E_{C,A}$ is non-simple, the coherent system $(E_{C,A},H^0(E_{C,A}))$ is not stable. We look at its maximal destabilizing sequence:
\begin{equation}\label{maximal}
0\to (E_1,H^0(E_1))\to (E_{C,A},H^0(E_{C,A}))\to (E_2,V_2)\to 0,
\end{equation}
The sheaf $E_2$ is torsion free and, being a quotient of $E$, it is globally generated by $V_2$ and satisfies $h^2(E_2)=0$, that is, $E_2$ is a g.LM bundle of type (II). Concerning the pair $(E_1,H^0(E_1))$, the statement follows directly from Lemma \ref{pizza} and short exact sequence (\ref{gg}) for $E_1$.
\end{proof}
\begin{rem}\label{minimo}
In the next section it will be convenient to replace the maximal destabilizing sequence (\ref{maximal}) with a {\em minimal destabilizing sequence}, where we require $(E_2,V_2)$ to be a stable quotient of $(E_{C,A},H^0(E_{C,A}))$ with $p_{(E_2,V_2)}$ minimal. In other words, the coherent system $(E_2,V_2)$ is any stable quotient of the minimal destabilizing quotient $(E_{C.A}/HN_{s-1}(E),H^0(E_{C,A})/H^0(HN_{s-1}))$. By construction, minimal destabilizing sequences are not unique. By Lemma \ref{pizza} and  short exact sequence (\ref{gg}) for $E_1$, the subbundle $E_1$ appearing in such a minimal destabilizing sequence  is an elementary modification of a g.LM bundle of type (I) and hence satisfies $h^2(E_1)=0$.
\end{rem}

\section{Linear series computing the Clifford index}
In this section, we describe complete linear series $A$ of type $g^r_d$ on a curve $C\subset S$ such that $\Cliff(C)=\Cliff(A)$ and $\rho(g,r,d)<0$. The following lemma explains how to recognize when a linear series on $C$ is cut out by the restriction of a line bundle on $S$ by only looking at the corresponding LM bundle.
\begin{lem}\label{estate}
Let $N_1\in\Pic(S)$ satisfy $h^0(N_1)\geq 2$ and $h^1(N_1)=0$. Also assume that the line bundle $N_2:=L\otimes N_1^\vee$ is globally generated and $h^1(N_2)=0$. Let $E$ be a g.LM bundle on $S$. Then, $E=E_{C,N_2\otimes\oo_C}$ for some smooth irreducible curve $C\in \vert L\vert$ if and only if $E$ is an extension of the form:
\begin{equation}\label{mentos}
0\to N_1\to E\to E_{C_2,\omega_{C_2}}\to 0,
\end{equation}
for some smooth irreducible curve $C_2\in\vert N_2\vert$.
\end{lem}
\begin{proof}
First, assume $E=E_{C,N_2\otimes\oo_C}$ for a smooth irreducible curve $C\in \vert L\vert$. Since $N_2$ is globally generated and $h^1(N_1)=0$, we obtain a diagram:
$$
\xymatrix{
&0\\
&N_1^\vee\ar[u]&&0\\
0\ar[r]&E_{C,N_2\otimes\oo_C}^\vee\ar[r]\ar[u]& H^0(N_2\otimes\oo_C)\otimes\oo_S\ar[r]&N_2\otimes\oo_C\ar[r]\ar[u]&0\\
0\ar[r]&K^\vee\ar[r]\ar[u]& H^0(N_2)\otimes\oo_S\ar[r]\ar[u]^{\simeq}&N_2\ar[r]\ar[u]&0.\\
&0\ar[u]&&N_1^\vee\ar[u]\\
&&&0\ar[u]\\
}
$$
\noindent One can easily check that $K$ is a globally generated vector bundle satisfying $h^i(K)=0$ for $i=1,2$, that is, $K$ is a LM bundle. Furthermore, one has $c_1(K)=c_1(N_2)$ and $c_2(K)=c_1(N_2)^2=2(\rk K-1)$. Corollary \ref{index} thus yields $K\simeq E_{C_2,\omega_{C_2}}$ for a smooth irreducible curve $C_2\in\vert N_2\vert$ and the "only if" part of the statement follows.

As for the converse implication, let $E$ sit in a short exact sequence as in (\ref{mentos}). Then, $E$ is locally free and globally generated, that is, $E=E_{C,A}$ for a smooth curve $C\in\vert L\vert$ and a line bundle $A\in \Pic^d(C)$ with $d=c_2(E)$. Look at the diagram:
$$
\xymatrix{
0\ar[r]&H^0(C,A)^\vee\otimes\oo_S\ar[r]&E\ar [r]^-\alpha&\omega_C\otimes A^\vee\ar[r]&0.\\
&&N_1\ar@{^{(}->}[u]_{\iota}&&
}
$$
Since $h^0(N_1)\geq 2$, then $\Hom(N_1,\oo_S)=0$ and $0\neq\alpha\circ \iota\in\Hom(N_1,\omega_C\otimes A^\vee)$. By adjunction, we get $h^0(A^\vee\otimes N_2\otimes \oo_C)>0$. It is easy to verify that $\deg (N_2\otimes\oo_C)=d$, hence $A$ is isomorphic to $N_2\otimes\oo_C$.
\end{proof}
\begin{rem}\label{ice}
The above proof also shows that, as soon as there exists $N\in\Pic(S)$ such that $h^0(N)\geq 2$ together with an injective morphism $N\hookrightarrow E_{C,A}$, one obtains that $h^0(A^\vee\otimes (L\otimes N^\vee)\otimes\oo_C)\neq 0$, i.e., the linear series $\vert A\vert$ is contained in $\vert (L\otimes N^\vee)\otimes\oo_C\vert$; this coincides with the restriction of $\vert L\otimes N^\vee\vert$ to $C$ if $h^1(N)=0$.
\end{rem}
\begin{ex}\label{pisano}
Assume the existence of an irreducible curve $B\subset S$ of genus $2$ such that $c_1(L)=2B$. For any smooth irreducible curve $C\in\vert L\vert$, the complete linear series $\oo_C(B)$ is of type $g^2_4$ and Lemma \ref{estate} exhibits the LM bundle $E_{C,\oo_C(B)}$ as an extension of the form:
$$
0\to\oo_S(B)\to E_{C,\oo_C(B)}\to E_{B,\omega_B}\to 0.
$$
Since $C$ has a $g^2_4$, it is then hyperelliptic and it turns out that $E_{C,g^1_2}=\oo_S(B)^{\oplus 2}$. This can be proved as follows. The unique $g^1_2$ on $C$ coincides with the line bundle \mbox{$\oo_C(B)(-x-y)$} for any two points $x,y\in C$ which are conjugate under the hyperelliptic involution. Hence, there exists a section $\alpha\in H^0(\oo_C(B))^\vee\subset H^0(E_{C,\oo_C(B)})$ such that $H^0(\oo_C(B))^\vee/\langle \alpha\rangle\simeq H^0(\oo_C(B)(-x-y))^\vee$, and the LM bundles $E_{C,\oo_C(B)}$ and $E_{C,g^1_2}$ fit in the exact sequence:
$$
0\to\oo_S\stackrel{\alpha}{\longrightarrow}E_{C,\oo_C(B)}\to E_{C,g^1_2}\to\oo_{x+y}\to0.
$$
We obtain an injection $\oo_S(B)\hookrightarrow E_{C,g^1_2}$, and its saturation gives a short exact sequence:
$$
0\to\oo_S(B+D)\to E_{C,g^1_2}\to\oo_S(B-D)\otimes I_\xi\to 0,
$$
where $l(\xi)=D^2$ and $D\geq 0$.  Since $2=c_2(E_{C,g^1_2})\geq B^2-D^2$, we have $D^2\geq 0$ and hence $D\cdot B>0$ by the Hodge Index Theorem. We show that $D=0$. Assume $D>0$; by Remark \ref{ice}, the $g^1_2$ on $C$ is contained in $\vert \oo_C(B-D)\vert$. Then, $\oo_C(B-D)$ is either a $g^1_2$ or a $g^1_3$ on $C$. The former case cannot occur because, if $D\sim D'+\Delta$ with $\Delta$ being the fixed part of $D$, then we would have $B\cdot D'=B\cdot D=1$; this is absurd as $B+D'$ is numerically $2$-connected. The latter case can also be excluded because $C\cdot (B-D)=2B\cdot (B-D)$ is even. Therefore, $D=0$ and $E_{C,g^1_2}=\oo_S(B)^{\oplus 2}$ as $h^1(\oo_S)=0$.

\end{ex}
\begin{ex}\label{mango}
Assume the existence of an irreducible elliptic curve $\Sigma\subset S$ such that $\Sigma\cdot c_1(L)=2$. Then, the restriction of $\Sigma$ induces a linear series of type $g^1_2$ on any curve $C\in \vert L\vert$. By the results in \cite{reid}, any hyperelliptic $K3$ surface is a ramified double cover of either $\mathbb{P}^2$ or the Hirzebruch surface $\mathbb{F}_n$ for $n\leq 4$. Therefore, one can write
$$
c_1(L)=D+(n+s)\Sigma,
$$
where $n,s$ are integers such that $s\geq 0$ and $0\leq n\leq 4$, and $D$ is an effective divisor satisfying $D^2=-2n$, and $D\cdot \Sigma=2$ and $H^0(\oo_S(D-\Sigma))=0$. Furthermore, every component $\Theta$ of $D$ satisfies $\Theta\cdot (D+n\Sigma)\geq 0$; in particular, for $n=0$, the divisor $D$ is an elliptic curve moving in a base point free pencil. An easy computation gives $g=n+2s+1$. Fix $r<g-1$; then, the line bundle $\oo_C(r\Sigma)$ gives the linear series $rg^{1}_{2}$ on $C$. By proceeding as in the first part of the proof of Lemma \ref{estate}, one can show that the LM bundle $E=E_{C,rg^1_2}$ sits in a short exact sequence of the form:
\begin{equation}\label{bene}
0\to \oo_S(D+(n+s-r)\Sigma)\to E\to \bigoplus^r\oo_S(\Sigma)\to 0.
\end{equation}
Note that the line bundle $ \oo_S(D+(n+s-r)\Sigma)$ is effective if and only if $r\leq n+s$.
\end{ex}
We can now prove the following strengthening of Theorem \ref{mannaggia}:
\begin{thm}\label{duro}
Let $A$ be a complete $g^r_d$ on a non-hyperelliptic and non-trigonal curve $C\subset S$ such that $d\leq g-1$, $\rho(g,r,d)<0$ and $\Cliff(A)=\Cliff(C)$. Assume $L:=\oo_S(C)$ is ample and the following condition is satisfied:
\begin{itemize}
\item[(*)] there is no irreducible elliptic curve $\Sigma\subset S$ such that $\Sigma\cdot C=4$ and no irreducible genus $2$ curve $B\subset S$ such that $B\cdot C=6$.
\end{itemize}
Then, one of the following occurs:
\begin{itemize}
\item[(i)] There exists a line bundle $M\in\Pic(S)$ adapted to $\vert L\vert$ such that $A\simeq M\otimes\oo_C$.
\item[(ii)] The line bundle $A$ satisfies $h^0(C,A)=2$ (i.e., $r=1$); furthermore, there exists a line bundle $M\in\Pic(S)$ adapted to $\vert L\vert$ such that $\vert A\vert$ is contained in the restriction of $\vert M\vert$ to $C$.
\end{itemize}

If condition (*) is not satisfied, then the following cases may also occur:
\begin{itemize}
\item[(iii)] There exists an irreducible curve $B$ of genus $2$ such that $C\sim 3B$ and $A$ is either a complete $g^2_6$ or a complete $g^3_8$; in both cases $\vert A\vert$ is contained in $\vert \oo_C(2B)\vert$.
\item[(iv)] There exist an irreducible curve $B$ of genus $2$ and an irreducible elliptic curve $\Sigma$ such that $B\cdot \Sigma=2$ and $C\sim 2B+\Sigma$; furthermore, $A$ is of type either $g^2_6$ or $g^3_8$ and $\vert A\vert$ is contained in $\vert \oo_C(B+\Sigma)\vert$.
\item[(v)] One has $C\sim \Sigma+\Sigma'+\Sigma''$ for three irreducible elliptic curves $\Sigma$, $\Sigma'$, $\Sigma''$ pairwise intersecting in two points, and $A$ is of type $g^2_6$; moreover,  the linear system $\vert A\vert$ is contained in $\vert \oo_C(\Sigma+\Sigma'')\vert$.
\item[(vi)] There exist two irreducible elliptic curves $\Sigma$, $\Sigma'$ and a divisor $D$ on $S$ satisfying \mbox{$D^2=-4$}, and $D\cdot \Sigma'=0$, and $D\cdot \Sigma=\Sigma\cdot\Sigma'=2$ such that $C\sim D+2\Sigma+\Sigma'$; furthermore, $A$ is of type $g^2_6$ and $\vert A\vert$ is contained in $\vert \oo_C(\Sigma+\Sigma')\vert$.
\item[(vii)] There are two elliptic curves $\Sigma$, $\Sigma'$ and a divisor $D$ such that $C\sim D+2\Sigma+\Sigma'$ with $D^2=-2$, and $D\cdot\Sigma'=0$, and $D\cdot \Sigma=\Sigma\cdot\Sigma'=2$; furthermore, $A$ is of type $g^2_6$ and $\vert A\vert$ is contained in $\vert\oo_C(\Sigma+\Sigma')\vert$.
\item[(viii)] There is an irreducible elliptic curve $\Sigma$ and a divisor $D$ such that $C\sim2D+3\Sigma$, and $D^2=-2$ and $D\cdot\Sigma=2$; moreover, $A$ is a complete $g^3_8$ and $\vert A\vert$ is contained in $\vert \oo_C(D+2\Sigma)\vert$.
\item[(ix)] There are two irreducible elliptic curves $\Sigma$, $\Sigma'$ satisfying $\Sigma\cdot \Sigma'=2$ and $C\sim 2\Sigma+2\Sigma'$; moreover, either $A$ is of type $g^2_6$ and $\vert A\vert$ is contained in $\vert\oo_C(\Sigma+\Sigma')\vert$, or $A$ is of type $g^3_8$ and $\vert A\vert$ is contained in $\vert\oo_C(2\Sigma+\Sigma')\vert$.
\end{itemize}
\end{thm}
Before starting the proof of Theorem \ref{duro}, notice that in all the cases (iii)--(ix) one has $\Cliff(C)=2$ and $g\leq 10$. In fact, condition (*) is automatically satisfied as soon as $\Cliff(C)> 2$.
\begin{proof}
Since $A$ computes $\Cliff(C)$, then $A$ is primitive and the LM bundle $E:=E_{C,A}$ is globally generated and non-simple (because $\rho(g,r,d)<0$). The condition $d\leq g-1$ forces $h^0(E)/\rk\,E=h^0(E)/(r+1)\geq 2$. We consider a minimal destabilizing sequence of the coherent system $(E,H^0(E))$:
\begin{equation}\label{minimal}
0\to (E_1,H^0(E_1))\to (E,H^0(E))\to (E_2,V_2)\to 0.
\end{equation}
In particular, $E_2$ is a g.LM bundle of type (II) and $(E_2,V_2)$ is stable (cf. Remark \ref{minimo}). There exists a commutative diagram:
$$
\xymatrix{
&0\\
&T\ar[u]&&0\\
0\ar[r]&\ar[r]E_1\ar[r]\ar[u]&E\ar[r]&E_2\ar[r]\ar[u]&0\\
0\ar[r]&\ar[r]\tilde{E_1}\ar[r]\ar[u]&E\ar[r]\ar@{=}[u]&\tilde{E_2}\ar[r]\ar[u]&0,\\
&0\ar[u]&&T\ar[u]\\
&&&0\ar[u]
}
$$
where $\tilde{E_1}$ is a g.LM bundle of type (I) satisfying $h^0(\tilde{E_1})=h^0(E_1)$, and $T$ is a pure sheaf of dimension $1$. Furthermore, Lemma \ref{pizza} gives $h^2(\tilde{E_1})=h^2(E_1)=0$ and hence $h^1(\tilde{E_2})=h^1(E_2)=0$. Being a quotient of $E$, the sheaf $\tilde{E_2}$ is globally generated and its torsion coincide with $T$. We set $r_i:=\rk E_i$ for $i=1,2$.

Note that both $\det\tilde{E_1}$ and $\det E_2$ are globally generated (cf. \cite[Lemma 3.3]{ciotola}), and $h^0(\det\tilde{E_2})\geq h^0(\det E_2)\geq 2$ and $h^0(\det E_1)\geq h^0(\det\tilde{E_1})\geq2$; in particular, both $\det(\tilde{E_1})\otimes\oo_C$ and $\det(E_2)\otimes\oo_C$ contribute to the Clifford index. By the Hirzebruch-Riemann-Roch Theorem applied to $T$, easy computations show that:
\begin{align}\label{primo}
\Cliff(E)=\Cliff(\tilde{E_1})+&\Cliff(E_2)+\nonumber\\
&+\left(c_1(E_2)\cdot c_1(T)+\frac{1}{2}c_1(T)^2-\chi(T)\right)+c_1(\tilde{E_1})\cdot c_1(\tilde{E_2})-2.
\end{align}
Since $h^1(\tilde{E_1})\geq \chi(T)$ and $C\cdot c_1(T)=(c_1(E_2)+c_1(\tilde{E_1})+c_1(T))\cdot c_1(T)\geq0$, one has:
\begin{align}\label{campane}
\Cliff(E)\geq\Cliff(\tilde{E_1})+&\Cliff(E_2)+\nonumber\\
&+\frac{1}{2}c_1(T)\cdot\left(c_1(E_2)-c_1(\tilde{E_1})\right)-h^1(\tilde{E_1})+c_1(\tilde{E_1})\cdot c_1(\tilde{E_2})-2,
\end{align}
and the inequality is strict as soon as $T\neq 0$ because $C$ is ample.
Now, we proceed by steps.\\\vspace{0.1cm}

\noindent{\em STEP I: If $c_1(\tilde{E_1})^2=c_1(E_2)^2=0$, then either $c_1(\tilde{E_1})=\oo_S(\Sigma_1)$, or $c_1(\tilde{E_2})=\oo_S(\Sigma_2)$ for some irreducible elliptic curve $\Sigma_i$.}

Let $c_1(\tilde{E_1})^2=c_1(E_2)^2=0$. Proposition \ref{elliptic} implies the existence of two irreducible elliptic curves $\Sigma_1$, $\Sigma_2$ such that $c_1(\tilde{E_1})=(r_1+h^1(E_1))\Sigma_1$ and $c_1(E_2)=r_2\Sigma_2$ (as $h^1(E_2)=h^2(E_1)=0$). Note that both $\oo_C(\Sigma_1)$ and $\oo_C(\Sigma_2)$ contribute to the Clifford index and:
\begin{eqnarray*}
\Cliff(\oo_C(\Sigma_1))&\leq&r_2\Sigma_1\cdot\Sigma_2+\Sigma_1\cdot c_1(T)-2,\\
\Cliff(\oo_C(\Sigma_2))&\leq&(r_1+h^1(E_1))\Sigma_1\cdot\Sigma_2+\Sigma_2\cdot c_1(T)-2.
\end{eqnarray*}
Using the fact that $c_1(\tilde{E_2})=c_1(E_2)+c_1(T)$, inequality (\ref{campane}) becomes:
\begin{align*}
\Cliff(E)\geq -2r_1-2r_2+2-h^1&(\tilde{E_1})+\frac{r_2}{2}\,\Sigma_2\cdot c_1(T)+\\&+\frac{r_1+h^1(\tilde{E_1})}{2}\,\Sigma_1\cdot c_1(T)+r_2(r_1+h^1(\tilde{E_1}))\Sigma_1\cdot \Sigma_2.
\end{align*}

First assume $\Sigma_1\cdot c_1(T)\geq\Sigma_2\cdot c_1(T)$; we will show that this implies $r_2=1$ and hence $h^1(\det E_2)=0$. Elementary manipulations give:
\begin{align}\label{bleau}
\Cliff(E)-\Cliff(\oo_C(\Sigma_2))\geq -2(r_1+r_2-2)&+\frac{r_1+r_2-2+h^1(\tilde{E_1})}{2}\,\Sigma_2\cdot c_1(T)+\nonumber\\
&+(r_2-1)(r_1+h^1(\tilde{E_1}))\Sigma_1\cdot \Sigma_2-h^1(\tilde{E_1}).
\end{align}
If $\Sigma_1\sim\Sigma_2$, then $T\neq 0$ (as $C$ is irreducible and hence it is not linearly equivalent to a multiple of $\Sigma_2$) and inequality (\ref{bleau}) is strict. In this case one easily obtains:
$$
\Cliff(E)-\Cliff(\oo_C(\Sigma_2))>-2(r-1)+\frac{r-1}{2}\,\Sigma_2\cdot c_1(T)+h^1(\tilde{E_1})\left(\frac{\Sigma_2\cdot c_1(T)-2}{2}\right)\geq0,
$$
where in the last inequality we have used that $\Sigma_2\cdot c_1(T)=\Sigma_2\cdot C\geq 4$ because $C$ is neither hyperelliptic nor trigonal. We have thus reached the contradiction $\Cliff(C)=\Cliff(E)>\Cliff(\oo_C(\Sigma_2))$.\\
If instead $\Sigma_1\not\sim\Sigma_2$, then $\Sigma_1+\Sigma_2$ is numerically $2$-connected and one obtains:
\begin{align}
\Cliff(E)-\Cliff(\oo_C(\Sigma_2))\geq2(r_1-1)(r_2-2)+(2r_2-3)h^1(\tilde{E_1});
\end{align}
this brings to the same contradiction as above as soon as $r_2\geq 2$. Indeed, if we had \mbox{$r_2=2$} and $\Cliff(E)=\Cliff(\oo_C(\Sigma_2))$, this would imply $T=0$ and $h^1(E_1)=0$; in particular, we would find $(E_2,V_2)=(\oo_S(\Sigma_2)^{\oplus 2}, H^0(\oo_S(\Sigma_2)^{\oplus 2}))$ and this is not possible since $(E_2,V_2)$ is stable by assumption.

Analogously, if $\Sigma_2\cdot c_1(T)>\Sigma_1\cdot c_1(T)$, one may show that
$$
\Cliff(E)-\Cliff(\oo_C(\Sigma_1))> 2(r_1-2)(r_2-1)+(2r_2-1)h^1(\tilde{E_1}),
$$
and hence conclude that $r_1=1$ and $h^1(\tilde{E_1})=0$.
\\\vspace{0.05cm}

In particular Step I, along with the strong version of Bertini's Theorem, implies that either $h^1(\det \tilde{E_1})=0$, or $h^1(\det E_2)=0$. Furthermore, by Corollary \ref{index} and Remark \ref{postino}, at least one of $\Cliff(\tilde{E_1})$ or $\Cliff(E_2)$ is nonnegative.\\\vspace{0.05cm}

\noindent{\em STEP II: If $c_1(\tilde{E_1})^2=0$, then $\tilde{E_1}=\oo_S(\Sigma_1)$ for an irreducible elliptic curve $\Sigma_1$.}

By Step I, we can assume $h^1(\det E_2)=0$ and hence $\Cliff(E_2)\geq 0$. Let $c_1(\tilde{E_1})^2=0$ and write $c_1(\tilde{E_1})=(r_1+h^1(\tilde{E_1}))\Sigma_1$ with $\Sigma_1$ as above. As in the previous step, $\oo_C(\Sigma_1)$ contributes to the Clifford index and $\Cliff(\oo_C(\Sigma_1))\leq\Sigma_1\cdot c_1(\tilde{E_2})-2$. Inequality (\ref{campane}) thus yields:
\begin{align*}
\Cliff(E)-\Cliff(\oo_C(\Sigma_1))\geq-2r_1+2+\frac{1}{2}&c_1(E_2)\cdot c_1(T)-\frac{r_1+h^1(\tilde{E_1})}{2}\,\Sigma_1\cdot c_1(T)+
\\&-h^1(\tilde{E_1})+(r_1+h^1(\tilde{E_1})-1)\Sigma_1\cdot c_1(\tilde{E_2});
\end{align*}
we recall that equality may hold only when $T=0$. Since $\det E_2$ is globally generated and $c_1(\tilde{E_2})=c_1(E_2)+c_1(T)$, one obtains:
\begin{align*}
\Cliff(E)-\Cliff(\oo_C(\Sigma_1))\geq\frac{r_1+h^1(\tilde{E_1})-2}{2}\,&\Sigma_1\cdot c_1(T)+(r_1-1)\left(\Sigma_1\cdot c_1(E_2)-2\right)+\\
&+h^1(\tilde{E_1})(\Sigma_1\cdot c_1(E_2)-1).
\end{align*}
Note that all divisors in $\vert \det E_2\otimes \oo_S(\Sigma_1)\vert$ are numerically $2$-connected; indeed, we have already showed that if $E_2=\oo_S(\Sigma_2)$, then $\Sigma_1\not\sim \Sigma_2$. Hence, $\Sigma_1\cdot c_1(E_2)\geq 2$ and, in order to avoid a contradiction, we conclude that $h^1(\tilde{E_1})=0$ and either $r_1=1$, or $T=0$ and $\Sigma_1\cdot c_1(E_2)=2$. The latter case can be excluded because $C$ is not hyperelliptic.\\\vspace{0.05cm}

Note that Step II implies $\Cliff(\tilde{E_1})\geq 2h^1(\tilde{E_1})$ and $h^1(\det\tilde{E_1})=0$. In particular, the line bundle $\det\tilde{E_1}$ is adapted to $\vert L\vert$ and the linear system $\vert \det \tilde{E_2}\otimes\oo_C\vert$ coincides with the restriction of $\vert \det \tilde{E_2}\vert$ to $C$. Furthermore, $\Cliff(\det\tilde{E_1}\otimes\oo_C)\leq c_1(\tilde{E_1})\cdot c_1(\tilde{E_2})-2$.\\\vspace{0.05cm}

\noindent{\em STEP III: If $c_1(E_2)^2=0$, then $E_2=\oo_S(\Sigma_2)$ for an irreducible elliptic curve $\Sigma_2$.}

The proof is very similar to that of Step II and consists in showing that $\Cliff(E)$ is strictly greater then $\Cliff(\oo_C(\Sigma_2))$ as soon as $r_2\geq 2$. We do not enter into computational details.\\\vspace{0.05cm}

Step III yields $\Cliff(E_2)\geq 0$ and $h^1(\det E_2)=0$. Moreover, one has the inequality \mbox{$\Cliff(\det E_2\otimes\oo_C)\leq c_1(E_1)\cdot c_1(E_2)-2$.}\\\vspace{0.05cm}

\noindent{\em STEP IV: One has $T=0$.}

Since $c_1(\tilde{E_1})\cdot c_1(\tilde{E_2})=c_1(E_1)\cdot c_1(E_2)-c_1(T)\cdot\left(c_1(E_2)-c_1(\tilde{E_1})\right)$, from (\ref{campane}) we obtain:
 \begin{eqnarray*}
\Cliff(E)&\geq& \Cliff(\tilde{E_1})+\Cliff(E_2)- h^1(\tilde{E_1})+\min\{c_1(\tilde{E_1})\cdot c_1(\tilde{E_2}),c_1(E_1)\cdot c_1(E_2)\}-2\\
&\geq&h^1(\tilde{E_1})+\min\{\Cliff(\det\tilde{E_1}\otimes\oo_C),\Cliff(\det E_2\otimes\oo_C)\}.
\end{eqnarray*}
If $T\neq 0$, then the first inequality would be strict, thus contradicting the fact that $\Cliff(C)=\Cliff(E)$.\\\vspace{0.05cm}

We can now rewrite (\ref{primo}) as
\begin{equation}
\Cliff(C)=\Cliff(E_1)+\Cliff(E_2)+\Cliff(\det{E_1}\otimes\oo_C),
\end{equation}
which forces $\Cliff(E_1)=\Cliff(E_2)=0$. In particular, this yields $h^1(E_1)=0$ and $V_2=H^0(E_2)$.\\\vspace{0.05cm}

\noindent{\em STEP V: If condition (*) is satisfied, then the sheaf $E_1$ is a line bundle and either $E_2$ is a line bundle as well, or $E_2=E_{C_2,\omega_{C_2}}$ for a smooth irreducible curve $C_2$ of genus $g_2\geq 2$. The same holds if condition (*) is not satisfied, except in the cases (iii)--(ix) of the statement.}
 
We apply Corollary \ref{index}. First of all, we exclude that $r_2>1$ and $E_2=E_{C_2,(r_2-1)g^1_2}$ for a smooth hyperelliptic curve $C_2$ of genus $g_2>r_2$. In fact, if this occurred, then the stability of $(E_2,H^0(E_2))$ would imply $\rho(g_2,r_2-1,2r_2-2)\geq 0$, hence a contradiction. Thus, either $r_2=1$ or $E_2=E_{C_2,\omega_{C_2}}$ for a smooth curve $C_2$ of genus $g_2\geq 2$. 

Now we remark that the fact that $(E_1,H^0(E_1))$ destabilizes $(E,H^0(E))$, together with the inequality $h^0(E)/(r+1)\geq 2$, prevents $E_1$ from being equal to $E_{C_1,\omega_{C_1}}$ for a smooth curve $C_1$ of genus $g_1=r_1\geq 2$. Assume instead $E_1=E_{C_1,(r_1-1)g^1_2}$ for a hyperelliptic curve $C_1$  of genus $g_1> r_1\geq 2$. In order to conclude the proof of Step V, it is enough to show that this occurs only in the cases (iii)-(ix) of the statement, where condition (*) is not satisfied.

First note that the inequality $h^0(E_1)/r_1\geq h^0(E)/(r+1)\geq 2$ implies $r_1\leq (g_1+1)/2$. 

If $C_1\sim 2B$ for a smooth irreducible curve $B$ of genus $2$, then either $r_1=3$ and $E_1=E_{C_1,\oo_{C_1}(B)}$, or $r_1=2$ and $E_1=\oo_S(B)^{\oplus 2}$ (cf. Example \ref{pisano}). In both cases we have an injective morphism $\oo_S(B)\hookrightarrow E$, hence by Remark \ref{ice} the linear system $\vert A\vert$ is contained in the restriction to $C$ of $\vert\det E_2\otimes\oo_S(B)\vert$. The line bundle $\det E_2\otimes\oo_C(B)$ contributes to the Clifford index and one can easily check that $\Cliff(\det E_2\otimes\oo_C(B))=B\cdot c_1(E_2)$. By the Hodge Index Theorem together with the inequality 
$$
\Cliff(C)=\Cliff(\det E_1\otimes\oo_C)=2B\cdot c_1(E_2)-2\leq B\cdot c_1(E_2),
$$
one of the following occurs:
\begin{itemize}
\item $c_1(E_2)=B$; then, $c_1(L)=3B$ and either $E_2=B$ or $E_2=E_{B,\omega_{B}}$. We have $g=10$ and $A$ is either a complete $g^3_8$ or a complete $g^2_6$; we are in case (iii).
\item $c_1(E_2)^2=0$ and $c_1(E_2)\cdot c_1(B)=2$ as every divisor in $\vert \det E_2\otimes\oo_S(B)\vert$ is numerically $2$-connected; then, Step III implies $E_2=\oo_S(\Sigma)$ for an irreducible elliptic curve $\Sigma\subset S$, and $c_1(L)=\Sigma+2B$. Hence, one has $g=9$ and $A$ is either of type $g^2_6$ or $g^3_8$ depending on $r_1$, that is, case (iv) occurs. 

\end{itemize}

Now, assume the existence of an irreducible elliptic curve $\Sigma$ such that $\Sigma\cdot C_1=2$ and let $C_1\sim D+(n+s)\Sigma$, with $n,s$ and $D$ as in Example \ref{mango}. Since $g_1\geq 3$, then $n+2s\geq 2$. Note that the line bundle $N_1:=\oo_S(D+(n+s-r_1+1)\Sigma)$ appearing in the short exact sequence (\ref{bene}) for $E_1$ is effective since $r_1-1\leq(g_1-1)/2= s+n/2$; we naturally obtain an injective morphism $N_1\hookrightarrow E$, and $\vert A\vert$ is thus contained in $\vert L\otimes N_1^\vee\otimes\oo_C\vert$. We use the inequality 
$$(D+(n+s-1)\Sigma)^2=2n+4s-4=2g_1-6\geq 0.$$
The line bundle $\oo_C(\Sigma)$ contributes to the Clifford index and one can easily check that $\Cliff(\oo_C(\Sigma))\leq\Sigma\cdot c_1(E_2)$. The following inequality is thus straightforward:
\begin{equation}\label{uffi}
0\geq\Cliff(C)-\Cliff(\oo_C(\Sigma))\geq(D+(n+s-1)\Sigma)\cdot c_1(E_2)-2.
\end{equation}
Also note that $c_1(E_2)\cdot D\geq 0$ and $c_1(E_2)\cdot \Sigma> 0$, because $\det E_2$ is globally generated and $c_1(E_2)\neq \Sigma$ (otherwise $C$ would be hyperelliptic). In particular, every divisor in $\vert \det E_2\otimes\oo_S(\Sigma)\vert$ is numerically $2$-connected and hence $c_1(E_2)\cdot \Sigma\geq 2$. Inequality (\ref{uffi}) then implies $n+s\leq 2$ and we only have the following possibilities.
\begin{itemize}
\item $n=0$ and $s=1$: Then, $C_1\sim \Sigma+\Sigma'$ for an elliptic curve $\Sigma'$ such that $\Sigma\cdot\Sigma'=2$; in particular $g_1=3$ and $r_1=2$. Note that $c_1(E_2)\neq\Sigma'$ and $c_1(E_2)\neq\Sigma$ because otherwise $C$ would be hyperelliptic. Therefore, $c_1(E_2)\cdot\Sigma'=2$ and we get equality in (\ref{uffi}) . The line bundle $\oo_C(\Sigma')$ contributes to the Clifford index and the inequality 
$$
2=c_1(E_2)\cdot\Sigma' \geq\Cliff(\oo_C(\Sigma'))\geq \Cliff(C)= \Cliff(\oo_C(\Sigma))=c_1(E_2)\cdot\Sigma
$$
 implies $c_1(E_2)\cdot\Sigma=2$. By applying the Hodge index Theorem to $\det E_1$ and $\det E_2$, we conclude that $c_1(E_2)^2\leq 4$.  If $c_1(E_2)^2=4$, then we have $c_1(E_2)=\Sigma+\Sigma'$. This case can be excluded since it is not compatible with the inequalities \mbox{$2=h^0(E_1)/r_1\geq h^0(E_2)/r_2$} and $d\leq g-1$. The same reason prevents the equality $c_1(E_2)^2=2$. It remains the possibility $c_1(E_2)^2=0$, that is, $c_1(E_2)=\Sigma''$ for an elliptic curve $\Sigma''$ satisfying $\Sigma\cdot\Sigma''=\Sigma'\cdot\Sigma''=2$. This implies $g=7$ (in fact, $C\sim \Sigma+\Sigma'+\Sigma''$) and $A$ is of type $g^2_6$, as in (v).
 \item $n=1$ and $s=0$: This case cannot occur since it contradicts $n+2s\geq 2$.
\item $n=2$ and $s=0$: Hence $C_1\sim D+2\Sigma$ has genus $g_1=3$ and $r_1=2$. Equalities $c_1(E_2)\cdot\Sigma =2$ and $c_1(E_2)\cdot D=0$, obtained from (\ref{uffi}), yield $C_1\cdot C_2=4$, where $C_2=c_1(E_2)$. If $E_2$ is a line bundle, then $h^0(E_2)=2$ because of the inequality $2=h^0(E_1)/r_1\geq h^0(E_2)$. Therefore, $E_2=\oo_S(\Sigma')$ for an irreducible elliptic curve $\Sigma'$, the curve $C\sim D+2\Sigma+\Sigma'$ has genus $7$ and $A$ is of type $g^2_6$, as in (vi). Assume instead that $E_2=E_{C_2,\omega_{C_2}}$ for a smooth curve $C_2$ of genus $g_2\geq 2$; easy computations give $g=g_2+6$ and $d=2g_2+4$, thus contradicting the inequality $d\leq g-1$.
\item $n=1$ and $s=1$: Then, $C_1\sim D+2\Sigma$ and one has $g_1=4$ and $r_1=2$. The Hodge Index Theorem, along with (\ref{uffi}), implies that either $E_2=\oo_S( \Sigma')$ for an irreducible elliptic curve $\Sigma'$ satisfying $\Sigma'\cdot D=0$ and $\Sigma'\cdot\Sigma=2$, or $C_2\sim D+\Sigma$. In the former case, $A$ is of type $g^2_6$ as in (vii). In the latter case, $E_2$ cannot be a line bundle because otherwise the inequality $5/2=h^0(E_1)/r_1\geq h^0(E_2)/r_2$ would be contradicted. Therefore, we obtain $E_2=E_{C_2,\omega_{C_2}}$ for a smooth curve $C_2$ and the line bundle $A$ is of type $g^3_8$, as in (viii).
\item $n=0$ and $s=2$: Hence, $C_1\sim \Sigma'+2\Sigma$ where $\Sigma'$ is an irreducible elliptic curve satisfying $\Sigma\cdot\Sigma'=2$; in particular, one has $g_1=5$ and $r_1\leq 3$. Inequality (\ref{uffi}) forces $\Sigma'\cdot c_1(E_2)=0$ and thus $E_2=\oo_S(\Sigma')$. We get $C\sim 2\Sigma+2\Sigma'$ and $A$ is of type $g^2_6$ or $g^3_8$, depending on $r_1$; in other words, case (ix) occurs.
\end{itemize}\vspace{0.5cm}

We can now conclude the proof of Theorem \ref{duro}. By Lemma \ref{estate}, one falls in case (i) with $M:=\det E_2$ as soon as $E_2=E_{C_2,\omega_{C_2}}$. On the other hand, if both $E_1$ and $E_2$ are line bundles, then trivially $r=1$ and one is in case (ii). Remark \ref{ice} then implies that $\vert A\vert$ is contained in the linear system $\vert\det E_2\otimes \oo_C\vert$, which coincides with the restriction of $\vert \det E_2\vert$ to $C$ by Step II. This concludes the proof.
\end{proof}
\begin{rem}\label{parto}
In case (ii) of Theorem \ref{duro}, assume the line bundle $M$ is very ample. Then, $M$ induces an embedding $S\hookrightarrow \mathbb{P}^n$ and one can easily show that $\vert A\vert$ is cut out by an ($n-2$)-plane which is ($2n-2$)-secant to $C$.
\end{rem}
When $L$ is ample, Theorem \ref{duro} is a refinement of the following result of  Green and Lazarsfeld \cite{green}, that we indeed reobtain as a corollary.
\begin{cor}
If $L\in\Pic(S)$ is ample, then all smooth irreducible curves in $\vert L\vert$ have the same Clifford index. Moreover, if this is strictly less then $[(g-1)/2]$, then there exists a line bundle $M$ on $S$ such that $M\otimes \oo_C$ computes the Clifford index of any curve $C\in\vert L\vert$.
\end{cor}
\begin{proof}
Let $C\in\vert L\vert$ be a curve of minimal Clifford index $\Cliff(C)<[(g-1)/2]$. We can assume that $C$ is neither hyperelliptic nor trigonal because otherwise the result follows directly from \cite{donat}. If $A$ is a complete $g^r_d$ on $C$ such that $d\leq g-1$ and $\Cliff(A)=\Cliff(C)$, then \mbox{$\rho(g,r,d)<0$} and one of the cases (i)--(ix) of Theorem \ref{duro} occurs. In case (i), the statement is straightforward. In all the other situations, Theorem \ref{duro} exhibits a line bundle $M$ on $S$ adapted to $\vert L\vert$ such that $\vert A\vert$ is contained in $\vert M\otimes\oo_C\vert$. One can easily check that $\Cliff(M\otimes\oo_C)=\Cliff(A)$.
\end{proof}

\section{The Donagi-Morrison Conjecture and secant varieties}
In this section, we study the Donagi-Morrison Conjecture. First of all, we exhibit a counterexample to Conjecture \ref{falsa} in the introduction.
\begin{cex}\label{fortuna}
Let $S$ be a $K3$ surface of genus $2$ such that \mbox{$\Pic(S)=\mathbb{Z}[B]$} and $B^2=2$. The vector bundle $E=\oo_S(B)^{\oplus 3}$ is globally generated and satisfies $h^i(E)=0$ for $i=1,2$, hence it coincides with the LM bundle $E_{C,A}$ associated with some primitive line bundle $A$ of type $g^2_6$ on some smooth irreducible curve $C\in\vert 3B\vert$.  The curve $C$ has genus $g=10$ and $\rho(10,2,6)<0$. By Remark \ref{ice}, the linear system $\vert A\vert$ is naturally contained in $\vert\oo_C(2B)\vert$, which is a complete $g^5_{12}$; note that $12>g-1=9$. In order to show that Conjecture \ref{falsa} fails for $A$, it is enough to show that $\vert A\vert$ is not contained in $\vert\oo_C(B)\vert$; since $\oo_C(B)$ is a complete $g^2_6$, this is equivalent to verifying that $A$ and $\oo_C(B)$ are not isomorphic. By Lemma \ref{estate}, if we had $E=E_{C,\oo_C(B)}$, then $E$ would fit in a short exact sequence like (\ref{mentos}) with $N_1=\oo_S(2B)$; this is a contradiction as $\Hom(\oo_S(2B),E)=0$.

Consider the rational map $h_E:G(3,H^0(E))\dashrightarrow \W^2_6(\vert 3B\vert)$ applying a general subspace $\Lambda$ to the pair $(C_\Lambda,A_\Lambda)$, where the evaluation map $ev_\Lambda:\Lambda\otimes\oo_S\to E$ drops rank along $C_\Lambda$ and has $\omega_{C_\Lambda}\otimes A_\Lambda^\vee$ as cokernel. We want to describe geometrically the pairs $(C_\Lambda,A_\Lambda)$ arising in this way. Note that we are in case (iii) of Theorem \ref{duro}.

The line bundle $B$ defines a $2$:$1$ cover $\pi:S\to\mathbb{P}^2$. A curve \mbox{$C\in\pi^*\vert \oo_{\mathbb{P}^2}(3)\vert\subset \vert 3B\vert$} is a $2$:$1$ cover of an elliptic curve $\Gamma$. Three points on $\Gamma$ define a complete $g^2_3$ on $\Gamma$, whose pullback to $C$ gives a complete linear series $A$ of type $g^2_6$; trivially, $A\simeq \oo_C(B)$ if and only if the three points on $\Gamma$ are aligned. The proof of Theorem \ref{duro} shows that, as soon as $A\in\Pic(C)$ is not isomorphic to $\oo_C(B)$, the LM bundle $E_{C,A}$ coincides with $E$, that is, $(C,A)\in \im(ev_\Lambda)$. Note that $\dim\pi^*\vert \oo_{\mathbb{P}^2}(3)\vert=9$ and any curve \mbox{$C\in\pi^*\vert \oo_{\mathbb{P}^2}(3)\vert$} has a 1-dimensional family of $g^2_6$ of the above form, hence the image of $h_E$ has dimension at least $10$. The fiber of $h_E$ over a point $(C_\Lambda,A_\Lambda)\in\im(h_E)$ is the projective space $\mathbb{P}\Hom(E,\omega_{C_\Lambda}\otimes A_\Lambda^\vee)\simeq\mathbb{P}(H^0(E\otimes E^\vee))$ and hence has dimension $8$. Since the Grassmannian $G(3,H^0(E))$ is $18$-dimensional, then $\overline{\im(h_E)}$ coincides with $W^2_6(\pi^*\vert \oo_{\mathbb{P}^2}(3)\vert)$ and moreover $\overline{\im(h_E)}\setminus \im(h_E)=\{(C,\oo_C(B))\,\textrm{s.t.}\, C\in\pi^*\vert \oo_{\mathbb{P}^2}(3)\vert\}$.
\end{cex}
\begin{rem}
One can easily check that in the cases (iii)--(ix) of Theorem \ref{duro} the line bundle $M$, for which $\vert A\vert$ is contained in the restriction of $\vert M\vert$ to $C$, satisfies the inequality $c_1(M)\cdot C>g-1$. As a consequence, by taking a general $K3$ surface $S$ containing a line bundle $L$ as in any of the cases (iii)--(ix), one can construct counterexamples to Conjecture \ref{falsa}.
\end{rem}
As already remarked in the introduction, in the class of counterexamples obtained above, the only contradicted part of Conjecture \ref{falsa} is the inequality \mbox{$c_1(M)\cdot C\leq g-1$}. In fact, our counterexamples still satisfy the modified Conjecture \ref{dm}, where such inequality is replaced by the condition that $M$ is adapted to $\vert L\vert$.

The following proposition reduces Conjecture \ref{dm} to the problem of exhibiting a suitable line bundle $N$ together with an injection $N\hookrightarrow E_{C,A}$.
\begin{prop}\label{carino}
Under the hypotheses of Conjecture \ref{dm}, let $A$ be primitive and assume the existence of a globally generated line bundle $N\in\Pic(S)$ such that $N$ is a saturated subsheaf of $E_{C,A}$. 

Then, Conjecture \ref{dm} holds with $M:=L\otimes N^\vee$ if $h^1(N)=0$, and with $M:=L(-\Sigma)$ if $c_1(N)=k\Sigma$ for an irreducible elliptic curve $\Sigma\subset S$ and an integer $k\geq 2$.
\end{prop}
\begin{proof}
If $h^1(N)=0$, then $M:=L\otimes N^\vee$ is adapted to $\vert L\vert$ and, by Remark \ref{ice}, $\vert A\vert$ is contained in the restriction of $\vert M\vert$ to $C$; it only remains to verify the inequality on the Clifford indices. Since $N\subset E:=E_{C,A}$ is saturated, the quotient $E/N$ is a g.LM bundle of type (II). A trivial computation gives:
$$
\Cliff(A)=c_1(N)\cdot c_1(E/N)+\Cliff(E/N)-2.
$$
If $h^1(\det E/N)=0$, then $\Cliff(E/N)\geq 0$ and hence:
$$\Cliff(\det E/N\otimes\oo_C)=c_1(N)\cdot c_1(E/N)-2\leq \Cliff(A).$$
If instead $h^1(\det E/N)\neq 0$, use the vanishing of $h^1(E/N)$ along with Proposition \ref{elliptic} and write $E/N=\oo_S(\Sigma')^{\oplus r}$  with $\Sigma'$ being an irreducible elliptic curve. The equality $h^0(\det E/N\otimes\oo_C)=r+1$ thus yields $$\Cliff(\det E/N\otimes \oo_C)=c_1(N)\cdot c_1(E/N)-2r=\Cliff(A).$$

Assume now $h^1(N)>0$, that is, $c_1(N)=k\Sigma$ with $k\geq 2$ and $\Sigma$ as in the statement. One obtains an injective morphism $\oo_S(\Sigma)\hookrightarrow E$ (that is not saturated). Since $C$ is not linearly equivalent to a multiple of $\Sigma$, then $L(-\Sigma)= \det E/N((k-1)\Sigma)$ satisfies $h^1(L(-\Sigma))=0$. Therefore, $L(-\Sigma)$ is adapted to $\vert L\vert$ and $\vert A\vert$ is contained in the restriction of $\vert L(-\Sigma)\vert$ to $C$. It remains to prove the following inequality:
\begin{equation}\label{pacchi}
\Cliff(\oo_C(\Sigma))=\Sigma\cdot \det E/N-2\leq \Cliff(E/N)+k(\Sigma\cdot \det E/N)-2=\Cliff(A).
\end{equation}
If $h^1(\det E/N)=0$, this is trivial since $\Cliff(E/N)\geq 0$. If instead $h^1(E/N)=0$, then as above $E/N=\oo_S(\Sigma')^{\oplus r}$ and $\Sigma'\not\sim\Sigma$. Inequaliy (\ref{pacchi}) becomes
$$
\Cliff(\oo_C(\Sigma))=r\Sigma\cdot \Sigma'-2\leq r.k(\Sigma\cdot \Sigma')-2r=\Cliff(A);
$$
this is satisfied because $\Sigma+\Sigma'$ is numerically $2$-connected and hence $\Sigma\cdot\Sigma'\geq 2$.
\end{proof}
In Donagi-Morrison's proof of the conjecture for $r=1$ (cf. \cite{donagi}), the line bundle $N$ is obtained as the kernel of an endomorphism of $E_{C,A}$ dropping the rank everywhere. In the case $r=2$ (cf. \cite{ciotola}), the role of the line bundle $N$ is played by the  maximal destabilizing sheaf $E_1$ of $E_{C,A}$ w.r.t. $\mu_L$-stability. In fact, a case-by-case analysis of all the possible types of Harder-Narasimhan and Jordan-H\"older filtrations shows that the condition $\rho(g,2,d)<0$ forces $E_1$ to have rank $1$ and the quotient $E/E_1$ to be $\mu_L$-stable. Unfortunately, these proofs cannot be extended to the general case.

We can now prove Theorem \ref{chiana}.
 \begin{proof}[Proof of Theorem \ref{chiana}]
 Up to replacing $A$ with $A+bs(\omega_{C\otimes A^\vee})$, we can assume that $A$ is primitive, that is, the LM bundle $E:=E_{C,A}$ is globally generated. Look at the maximal destabilizing sequence of $(E,H^0(E))$:
 $$
 0\to (E_1,H^0(E_1))\to (E,H^0(E))\to (E_2,V_2)\to 0,
 $$
 where $E_2$ is a g.LM bundle of type (II) and rank $r_2$, while $E_1$ is generically generated by its global sections and has rank $r_1$. Up to replacing $(E_1,H^0(E_1))$ with a stable coherent subsystem appearing in its Jordan-H\"older filtration, we can assume that $(E_1,H^0(E_1))$ is stable. Since $S$ does not contain any ($-2$)-curve, we can also assume that $E_1$ is a g.LM bundle of type (I); indeed, if this were not the case, Proposition \ref{generated} would then provide a nef line bundle $N\hookrightarrow E_1$ and, after taking the saturation of $N\subset E$, the statement would be a straight consequence of Proposition \ref{carino}.  By \cite[Lemma 3.3]{ciotola}, both $\det E_1$ and $\det E_2$ are globally generated.
 
 Note that the conjecture follows directly from Proposition \ref{carino} as soon as $r_1=1$.

Let now $r_1>1$. If $c_1(E_1)^2=0$, Proposition \ref{elliptic} yields the existence of an irreducible elliptic curve $\Sigma_1\subset S$ along with a short exact sequence
 $$
 0\to\oo_S^{\oplus h^1(E_1)}\to \oo_S(\Sigma_1)^{\oplus h^1(E_1)+r_1}\to E_1\to 0,
 $$
 and hence $\oo_S(\Sigma_1)\hookrightarrow E_1$. Therefore, $\oo_S(\Sigma_1)$ is a subsheaf of $E$ and, after replacing it with its saturation, Conjecture \ref{dm} follows from Proposition \ref{carino}.

It remains to deal with the case $c_1(E_1)^2>0$, which yields $h^1(\det E_1)=0$. By Proposition \ref{general}, given a general subspace \mbox{$V_1\in G(r_1-1, H^0(E_1))$}, the degeneracy locus of the map $ev_1:V_1\otimes\oo_S\to E_1$ is a $0$-dimensional subscheme $\xi$ consisting of some distinct points which do not lie in $\mathrm{Sing}(E_2)$. We obtain the following diagram:
 $$
\xymatrix{
&&0&0&0\\
&0\ar[r]&E_2\ar[r]\ar[u]&E_2'\ar[r]\ar[u]&\kappa\ar[r]\ar[u]&0\\
0\ar[r]& V_1\otimes\oo_S\ar[r]&E\ar[r]\ar[u]&E'\ar[r]\ar[u]&\oo_\xi\oplus\kappa\ar[r]\ar[u]&0\\
0\ar[r]& V_1\otimes\oo_S\ar[r]\ar@{=}[u]&E_1\ar[r]\ar[u]&\det E_1\ar[r]\ar[u]&\oo_\xi\ar[r]\ar[u]&0.\\
&&0\ar[u]&0\ar[u]&0\ar[u]\\
}
$$
The sheaf $\kappa$ is $0$-dimensional and its support is contained in $\mathrm{Sing}(E_2)$. By construction, $E'$ is locally free and $E_2'$ has no torsion. Since $h^1(E_2)=h^2(E_1)=0$, then $E_2'$ is globally generated and satisfies $h^i(E_2')=0$ for $i=1,2$, i.e., it is a g.LM bundle of type (II). Since both $\det E_1$ and $E_2'$ are globally generated, the same holds for $E'$; moreover, $E'$ fulfills the condition $h^i(E')=0$ for $i=1,2$, hence it is the LM bundle associated with a primitive line bundle on a smooth curve in the linear system $\vert L\vert$. 

The image of $H^0(E)/V_1$ in $H^0(E')$, that we denote by $V'$, generates the bundle $E'$ outside of $\xi\cup\mathrm{Supp}(\kappa)$. Furthermore, $H^0(E')/V'$ surjects onto $H^0(\kappa)$ because both $h^1(E_2)=0$ and $h^1(\det E_1)=0$. If $V_2\in G(\rk E',V')$ is general, then the evaluation map $ev_2:V_2\otimes\oo_S\to E'$ is injective and we get a diagram:
$$
\xymatrix{
&&0&0\\
&0\ar[r]&B_1\ar[r]\ar[u]&B_1'\ar[r]\ar[u]&\oo_\xi\oplus\kappa\ar[r]&0\\
0\ar[r]&V_1\otimes\oo_S\ar[r]& E\ar[r]\ar[u]&E' \ar[u]\ar[r]&\oo_\xi\oplus\kappa\ar[r]\ar@{=}[u]&0;\\
0\ar[r]&V_1\otimes\oo_S\ar[r]\ar@{=}[u]& (V_1\oplus V_2)\otimes\oo_S\ar[r]\ar[u]& V_2\otimes\oo_S\ar[r]\ar[u]&0&\\
&&0\ar[u]&0\ar[u]
}
$$
here, $B_1$ and $B'_1$ are pure sheaves of dimension $1$ supported on a curve $X\in\vert L\vert$, which is integral but possibly singular along $\mathrm{Supp}(\kappa)$. Both $B_1$ and $B_1'$ are globally generated and, as sheaves on $X$, they have rank $1$ and no torsion. One can show that $H^0(B_1')$ surjects onto $H^0(\kappa)$. We define $B:= \mathcal{E}xt^1(B_1,\oo_S)$ and $B':=\mathcal{E}xt^1(B_1',\oo_S)$. It turns out that $H^0(B)\simeq H^1(B_1)^\vee$ and $H^0(B')\simeq H^1(B_1')^\vee$ (cf. \cite[Lemma 2.3]{gomez}), hence $V_1\oplus V_2\simeq H^0(B)^\vee$ and $V_2\simeq H^0(B')^\vee$; in particular, one has $h^0(B)=r+1$ and $h^0(B')=r_2+1$. By dualizing the second column in the diagram, one easily shows that $E_{X,B}= E$.
Look at the diagram:
$$
\xymatrix{
&&0&0\\
&&\oo_\xi\ar[r]^-\simeq\ar[u]&\mathcal{E}xt^2(\oo_{\xi},\oo_S)\ar[u]\\
0\ar[r]&B'\ar[r]&B\ar[r]\ar[u]&\mathcal{E}xt^2(\oo_{\xi}\oplus\kappa,\oo_S)\ar[u]\ar[r]&0\\
0\ar[r]&B'\ar[r]\ar[u]^\simeq&B(-\xi)\ar[r]\ar[u]&\mathcal{E}xt^2(\kappa,\oo_S)\ar[u]\ar[r]&0.\\
&&0\ar[u]&0\ar[u]
}
$$

Note that $H^0(\mathcal{E}xt^2(\kappa,\oo_S))\simeq H^0(\kappa)^\vee\hookrightarrow H^0(B_1')^\vee\simeq H^1(B')$ (cf. \cite[p. 39]{friedman} for the first isomorphism). Hence, we have $H^0(B')\simeq H^0(B(-\xi))$ and $\xi\in V^{e-f}_e(B)$ for $e=c_2(E_1)$ and $f=c_2(E_1)-r_1+1$. Easy computations give
\begin{equation}\label{end}
\mathrm{expdim}\,V^{e-f}_e(B)=-r_2c_2(E_1)+(r_2+1)(r_1-1)\leq -(r_1-1)(r_2-1),
\end{equation}
where the inequality follows from the fact that $\Cliff(E_1)\geq 0$. Therefore, the variety $V^{e-f}_e(B)$ has negative expected dimension unless $r_1=r$ (or equivalently, $r_2=1$) and $\Cliff(E_1)=0$; we can exclude this case by applying Corollary \ref{index}. Indeed, the bundle $E_1$ cannot be the LM bundle $E_{C_1,\omega_{C_1}}$ for a smooth curve $C_1$ of genus $g_1\geq 2$ because this is not compatible with inequalities $h^0(E_1)/r_1\geq h^0(E)/(r+1)\geq 2$. Analogously, if we had $E_1=E_{C_1,(r_1-1)g^1_2}$ for a hyperelliptic curve $C_1$ of genus $g_1>r_1$, then inequality $\rho(g_1,r_1-1,2r_1-2)< 0$ would hold thus contradicting the stability of $(E_1,H^0(E_1))$.

In summary, as soon as $r_1>1$, we have found a pair $(X,B)$ as above such that $\mathrm{expdim}\,V^{e-f}_e(B)<0$; as a consequence, the initial pair $(C,A)$ has some unexpected secant varieties up to deformation, in contradiction with the hypotheses. Hence, the only possibility is $r_1=1$ and this concludes the proof.
\end{proof}

We conclude with two remarks concerning the condition of {\em having some unexpected varieties up to deformation}. By definition, this involves integral curves $X\in\vert L\vert$ which might be singular and sheaves $B\in \overline {J}^d(X)$. First, we explain how to deform such a pair $(X,B)$ to another pair $(\tilde{C},\tilde{A})$ such that $\tilde{C}\in\vert L\vert$ is smooth and $\tilde{A}\in\Pic^d(\tilde{C})$ has some special properties.
\begin{rem}\label{deformo}
We will show that, if $(C,A)$ has some unexpected secant varieties up to deformation, then it can be deformed to a pair $(\tilde{C},\tilde{A})$ such that $\tilde{C}\in\vert L\vert$ is smooth and irreducible and $\tilde{A}\in\Pic^d(\tilde{C})$ fulfills either condition
\begin{itemize}
\item[(i)] $h^0(\tilde{A})>h^0(A)=r+1$
\end{itemize}
or condition
\begin{itemize}
\item[(i')] $V^{e-h}_e(\tilde{A})\neq \emptyset$ and $\mathrm{expdim}\,V^{e-h}_e(\tilde{A})<0$ for some integers $0\leq h<e$.
\end{itemize}

Note that condition (i) implies that the line bundle $\tilde{A}$ is of type $g^{s}_d$ with $s>r$; in particular, one has $\rho(g,s, d)<\rho(g,r,d)<0$. 

The construction of the pair $(\tilde{C},\tilde{A})$ uses the results in \cite{gomez}. By assumption, there is an integral curve $X\in \vert L\vert$ together with a sheaf $B\in \overline {J}^d(X)$ that satisfies $h^0(B)=r+1$  and sits in a short exact sequence
\begin{equation}\label{uffino}
0\to B'\to B\stackrel{q}{\longrightarrow} Q\to 0,
\end{equation}
where $q\in V_e^{e-f}(B)$ for some integers $e,f$ such that $\mathrm{expdim}\,V^{e-f}_e(B)<0$. We may assume $f=h^0(B')-h^0(B)+e$.

By \cite[Proposition 2.5 and Section 3]{gomez}, there exist a family of curves $\mathcal{C}$ parametrized by a connected (but not necessarily complete) curve $T$, a connected (possibly neither irreducible nor complete) curve $Y$ together with a map $p:Y\to T$, and a sheaf $\mathcal{B}'$ on $S\times Y$ flat over $Y$ such that:
\begin{itemize}
\item[(a)] $\mathcal{C}_{t_0}\simeq X$ for some $t_0\in t$, and for every $t\neq t_0$ the curve $\mathcal{C}_t\in \vert L\vert$ is smooth and irreducible;
 \item[(b)] there is a component $Y_1$ of $Y$ such that $p\vert_{Y_1}$  is a finite cover, and $p$ contracts all the other components of $Y$ to $t_0$;
\item[(c)] for all $y\in Y$ the sheaf $B'_y$ induced by $\mathcal{B}'$ on the fiber of $S\times Y\to Y$ over $y$ is a torsion free sheaf of rank $1$ on $\mathcal{C}_{p(y)}$, and $B'_{y_0}\simeq B'$ for some $y_0\in p^{-1}(t_0)$;
\item[(d)] $h^0(B'_y)\geq h^0(B')$ for all $y\in Y$.
\end{itemize}

By \cite[Proposition 3.6]{gomez}, up to replacing $Y$ with another curve $Y'$ (which maps to $Y$ and satisfies an analogue of (b))  and $\mathcal{B}'$ with its pullback to $S\times Y'$, we find sheaves $\mathcal{B}$ and $\mathcal{Q}$ on $S\times Y$ flat over $Y$ sitting in an exact sequence
\begin{equation}
0\to \mathcal{B}'\to \mathcal{B}\to \mathcal{Q}\to 0;
\end{equation}
this induces, for all $y\in Y$, a short exact sequence
$$
0\to B'_y\to B_y\to Q_y\to 0,
$$
which is equivalent to (\ref{uffino}) when $y=y_0$. Set $\tilde{C}:=\mathcal{C}_{t_1}$ for a general $t_1\in T$ and $\tilde{A}:=B_{y_1}$ for some $y_1\in p^{-1}(t_1)$; then, $B'_{y_1}=\tilde{A}(-D)$ for an effective divisor $D\in\tilde{C}_e$. If $h^0(\tilde{A})>h^0(B)$, then we are in case (i). If instead $h^0(\tilde{A})\leq h^0(B)$, then $D\in V^{e-h}_e(\tilde{A})$ for $h=e-h^0(\tilde{A})+h^0(\tilde{A}(-D))$. Since $h\geq f$, we get
$$
\mathrm{expdim}\,V^{e-h}_e(\tilde{A})=e-h\cdot h^0(\tilde{A}(-D))\leq e-f\cdot h^0(B')= \mathrm{expdim}\,V^{e-f}_e(B)<0,
$$ 
and condition (i') is satisfied.
\end{rem}
Next remark explains the implications of our condition in terms of points on the surface $S$.
\begin{rem}
As usual, let $A$ be a primitive $g^r_d$ on a curve $C\subset S$ and set $E:=E_{C,A}$. Then, for a general subspace $W\in G(r,H^0(A)^\vee)\subset G(r, H^0(E))$ we obtain a short exact sequence:
$$
0\to W\otimes \oo_S\to E\to L\otimes I_\zeta\to 0,
$$
where $I_\zeta$ is the ideal sheaf of a $0$-dimensional subscheme of $S$ of length $d=c_2(E)$. Assume that $L$ is very ample, i.e., it gives an embedding $S\hookrightarrow \mathbb{P}^{h^0(L)-1}=\mathbb{P}^g$. Since $h^0(L\otimes I_\zeta)=g-d+r+1$, then $\zeta$ is cut out by a ($d-r-1$)-plane which is $d$-secant to the surface $S$. 

Now, assume that $(C,A)$ has some unexpected secant varieties up to deformation and let $(X,B)$ be the pair given by the definition. As above, the subspaces lying in the Grassmannian $G(r,H^0(B)^\vee)\subset G(r, H^0(E))$ correspond to an $r$-dimensional family of $d$-secant ($d-r-1$)-planes to $S\subset\mathbb{P}^g$. One can show that the non-emptiness of $V^{e-f}_e(B)$ is equivalent to the requirement that at least one of these planes contains a ($d-r-1-f$)-plane which is ($d-e$)-secant to $S$. 

An interpretation of the condition $\mathrm{expdim}\,V^{e-f}_e(B)<0$ in terms of points on $S$ would also be interesting.
\end{rem}

\pagestyle{myheadings}
\markboth{ANDREAS LEOPOLD KNUTSEN,\, MARGHERITA LELLI-CHIESA}{}
\newpage

\appendice{ A counterexample to Conjecture \ref{dm}}
\begin{center}
{\normalsize
ANDREAS LEOPOLD KNUTSEN,\, MARGHERITA LELLI-CHIESA
\par
}
\end{center}\vspace{0.5cm}

In this Appendix we provide counterexamples to Conjecture \ref{dm} and show that the hypothesis in Theorem \ref{chiana} concerning secant varieties cannot be avoided. 

Let $S$ be a $K3$ surface of even genus $p\geq 2$ such that \mbox{$\Pic(S)=\mathbb{Z} [H]$} with $H^2=2p-2$. Lazarsfeld's Theorem (cf. \cite{lazarsfeld}) implies that every smooth irreducible curve $H_0\in\vert H\vert$ has gonality $k=(p+2)/2$. By \cite[Theorem 3]{mukai1}, the LM bundle $E:=E_{H_0,B_0}$ associated with a complete linear series $B_0$ of type $g^1_k$ on a smooth curve $H_0\in\vert H\vert$ does not depend on the choice of the pair $(H_0,B_0)$. Furthermore, $E$ is both globally generated and stable because the Picard group of $S$ is cyclic and $\Hom(\oo_S(H),E)=0$. The vector bundle $E\oplus E$ is generated by its global sections and satisfies $h^i(E\oplus E)=0$ for $i=1,2$, hence it coincides with the LM bundle associated with a complete and primitive linear series $A$ of type $g^3_{3p}$ on a smooth irreducible curve $C\in \vert2 H\vert$, which has genus $g:=4p-3$. Notice that $3p\leq g-1$ as soon as $p\geq 4$. An easy computation gives $\rho(g,3,3p)=-3$ and we claim that both Conjectures \ref{falsa} and \ref{dm} fail for the pair $(C,A)$. Notice that the only sheaf which might play the role of the line bundle $M$ appearing in the statement of the conjectures is $\oo_S(H)$. The condition that the linear system $\vert A\vert$ is contained in the restriction of $\vert H\vert$ to $C$ implies the inequality $h^0(A^\vee \otimes \oo_C(H))>0$. Tensoring by $\oo_S(-H)$ the short exact sequence
$$
0\to H^0(A)^\vee\otimes \oo_S\to E\oplus E\to \oo_C(2H)\otimes A^\vee\to 0,
$$
one finds that $h^0(A^\vee \otimes \oo_C(H))=2h^0(E(-H))=0$ and this proves our claim.

Now we show that the pair $(C,A)$ has some unexpected secant varieties up to deformation; as a consequence, the result of Theorem \ref{chiana} is as good as possible. Take a general section $\alpha\in H^0(E)$ and consider the diagram:
 $$
\xymatrix{
&&0&0&\\
&&E\ar@{=}[r]\ar[u]&E\ar[u]\\
0\ar[r]& \oo_S\ar[r]&E\oplus E\ar[r]\ar[u]& E'\ar[r]\ar[u]& \oo_\xi\ar[r]&0\\
0\ar[r]& \oo_S\ar[r]^\alpha\ar@{=}[u]&E\ar[r]\ar[u]&\oo_S(H)\ar[r]\ar[u]&\oo_\xi\ar[r]\ar@{=}[u]&0,\\
&&0\ar[u]&0\ar[u]&\\
}
$$
where $\xi\subset S$ is a $0$-dimensional subscheme consisting of $k$ distinct points and $E'$ is a rank $3$ vector bundle. As in the proof of Theorem \ref{chiana}, one finds a deformation $(\tilde{C},\tilde{A})$ of the pair $(C,A)$ such that $E\oplus E=E_{\tilde{C},\tilde{A}}$ and $E'=E_{\tilde{C}, \tilde{A}(-\xi)}$; since $\xi$ is curvilinear, one can choose $\tilde{C}$ to be smooth. The equalities $h^0(\tilde{A}(-\xi))=3=h^0(\tilde{A})-1$ prevent $\tilde{A}$ from being very ample; in other words, one has $\xi\in V^1_{(p+2)/2}(\tilde{A})$ and can easily check that $\mathrm{expdim}\, V^1_{(p+2)/2}(\tilde{A})=-p+1<0$.
\begin{rem}
The above counterexample contradicts part (i) of Conjecture \ref{dm}, that is, $A$ is not contained in the restriction of $\vert M\vert$ to $C$ for any line bundle $M$ which is adapted to $\vert 2H\vert$.
\end{rem}
\begin{rem}
Similarly, one can construct counterexamples to Conjecture \ref{dm} in cases where the line bundle \mbox{$L=\oo_S(C)$} is primitive. For instance, take a $K3$ surface $S$ with \mbox{$\Pic(S)=\mathbb{Z}[H]\oplus \mathbb{Z}[D]$} where $H^2=2r-2\geq 4$, $D^2=2p-2\geq 4$ and $H\cdot D=d$ such that both $p$ and $r$ are even and $p+r-d=0$. Necessary and sufficient conditions on $r$ and $p$ for the existence of such an $S$ with $H$ very ample and $D$ smooth and irreducible are given in \cite{andreas}. However, for our purposes it is enough just to assume that the integers $r$, $p$ are such that there are no divisors of self-intersection $0$ or $-2$ in the lattice; such a $K3$ surface exists as explained in \cite{andreas} (since the inequality $d^2-4(r-1)(p-1)>0$ is satisfied because of the condition $p+r-d=0$) and $D$ and $H$ are automatically very ample. By the results in \cite{arap}, the integers $r$, $p$, can be chosen in infinitely many ways so that the linear system $\vert H\vert$ does not contain any reducible or non-reduced member (we leave the details to the reader) and similarly for $\vert D\vert$ by symmetry. In particular, a general curve $H_1\in\vert H\vert$ has maximal gonality $k_1=(H^2+6)/4$, and there are no injective morphisms $N\hookrightarrow E_1$ from an effective line bundle $N$ to the LM bundle $E_1:=E_{H_0,A_0}$ associated with a linear series $A_0$ of type $g^1_{k_1}$ on $H_0$, and analogously for $\vert D\vert$. We set $k_2=(D^2+6)/4$ and denote by $E_2$ the LM bundle associated with a $g^1_{k_2}$ on a general curve $D_2\in\vert D\vert$.

Consider the vector bundle $E_1\oplus E_2$; this coincides with the LM bundle associated with a primitive linear series $A$ of type $g^3_{k_1+k_2+d}$ on a smooth irreducible curve \mbox{$C\in\vert H+D\vert$}. Trivial computations give $$\rho(g(C),3, k_1+k_2+d)=-1.$$ Since there are no morphisms from an effective line bundle $N$ to the LM bundle \mbox{$E_1\oplus E_2$}, one shows, arguing as above, that Conjecture \ref{dm} fails for the pair $(C,A)$.
\end{rem}

\vfill

\end{document}